\documentclass[conference]{IEEEtran}
\usepackage{times}

\usepackage[numbers]{natbib}
\usepackage{multicol}
\usepackage[bookmarks=true]{hyperref}

\usepackage{graphicx} 

\usepackage[table]{xcolor}
\usepackage{amsmath}
\usepackage{amssymb}
\usepackage{amsthm}
\usepackage{mathtools}
\usepackage{derivative}
\usepackage{url}
\usepackage{hyperref}
\usepackage[ruled,linesnumbered]{algorithm2e}
\usepackage[capitalize,nameinlink]{cleveref}
\usepackage{bm}
\usepackage{nicematrix}
\usepackage{caption}
\usepackage{subcaption}
\usepackage{wrapfig}
\usepackage{float}
\usepackage{framed}
\usepackage{cancel}

\hypersetup{
    colorlinks=true,
    linkcolor=blue,
    filecolor=magenta,
    urlcolor=cyan,
}

\usepackage{hhline}  
\usepackage{makecell}
\usepackage{multirow}
\definecolor{LightGreen}{rgb}{.9, 1, .9}
\definecolor{LightRed}{rgb}{1, .9, .9}
\newcolumntype{g}{>{\columncolor{LightGreen}}c}
\newcolumntype{r}{>{\columncolor{LightRed}}c}

\usepackage{siunitx}
\newtheorem{theorem}{Theorem}

\newtheorem{lemma}[theorem]{Lemma}
\newtheorem{corollary}[theorem]{Corollary}

\newtheorem{proposition}[theorem]{Proposition}
\newtheorem{assumption}[theorem]{Assumption}

\newtheorem*{lemma*}{Lemma}

\newtheorem{manualtheoreminner}{Theorem}
\newenvironment{manualtheorem}[1]{%
  \renewcommand\themanualtheoreminner{#1}%
  \manualtheoreminner
}{\endmanualtheoreminner}

\makeatletter
\def\thm@space@setup{\thm@preskip=5pt
\thm@postskip=0pt}
\makeatother
\theoremstyle{definition}
\newtheorem{problem}{Problem}

\theoremstyle{remark}
\newtheorem{remark}{Remark}

\newcommand{\R}{\mathbb{R}}

\DeclareMathOperator{\diag}{diag}
\DeclareMathOperator{\ones}{ones}

\newcommand{\defeq}{\vcentcolon=}
\newcommand{\eqdef}{=\vcentcolon}
\newcommand{\norm}[1]{\left\lVert#1\right\rVert}

\newcommand{\normt}[1]{\lVert#1\rVert}

\DeclareMathOperator*{\argmin}{arg\,min}

\usepackage[nolist]{acronym}
\begin{acronym}
    \acro{ML}{machine learning}
    \acro{OC}{optimal control}
    \acro{DOC}{differentiable optimal control}
    \acro{IOC}{inverse optimal control}
    \acro{RL}{reinforcement learning}
    \acro{IRL}{inverse reinforcement learning}
    \acro{PMP}{Pontryagin's maximum (minimum) principle}
    \acro{KKT}{Karush-Kuhn-Tucker}
    \acro{LQ}{linear-quadratic}
    \acro{LQR}{linear-quadratic regulator}
    \acro{iLQR}{iterative linear-quadratic regulator}
    \acro{VJP}{vector-Jacobian product}
    \acro{DDP}{differential dynamic programming}
    \acro{MPC}{model predictive control}
    \acro{PDP}{Pontryagin differentiable programming}
    \acro{Diff-MPC}{differentiable model predictive control}
    \acro{DBaS}{discrete barrier state}
    \acro{DBaS-DDP}{discrete barrier state differential dynamic programming}
    \acro{NT-MPC}{nonlinear tube-based MPC}
    \acro{DT-MPC}{differentiable tube-based MPC}
    \acro{AD}{automatic differentiation}
    \acro{IFT}{implicit function theorem}
\end{acronym}

\begin{document}

\title{
Differentiable Robust Model Predictive Control
}


\author{\IEEEauthorblockN{Alex Oshin}
\IEEEauthorblockA{Georgia Institute of Technology\\
\texttt{alexoshin@gatech.edu}}
\and
\IEEEauthorblockN{Hassan Almubarak}
\IEEEauthorblockA{Georgia Institute of Technology\\
\texttt{halmubarak@gatech.edu}}
\and
\IEEEauthorblockN{Evangelos A. Theodorou}
\IEEEauthorblockA{Georgia Institute of Technology\\
\texttt{evangelos.theodorou@gatech.edu}}}



%

\maketitle
\vspace*{-1.0cm}
\begin{abstract}
Deterministic model predictive control (MPC), while powerful, is often insufficient for effectively controlling autonomous systems in the real-world. Factors such as environmental noise and model error can cause deviations from the expected nominal performance. Robust MPC algorithms aim to bridge this gap between deterministic and uncertain control. However, these methods are often excessively difficult to tune for robustness due to the nonlinear and non-intuitive effects that controller parameters have on performance. To address this challenge, we first present a unifying perspective on differentiable optimization for control using the implicit function theorem (IFT), from which existing state-of-the art methods can be derived. Drawing parallels with differential dynamic programming, the IFT enables the derivation of an efficient differentiable optimal control framework. The derived scheme is subsequently paired with a tube-based MPC architecture to facilitate the automatic and real-time tuning of robust controllers in the presence of large uncertainties and disturbances. The proposed algorithm is benchmarked on multiple nonlinear robotic systems, including two systems in the MuJoCo simulator environment and one hardware experiment on the Robotarium testbed, to demonstrate its efficacy.
\end{abstract}

\IEEEpeerreviewmaketitle

\section{Introduction}
\label{sec:introduction}

\Ac{MPC} is a predominant approach for real-time optimization-based motion planning and control across robotics~\citep{poignet2000nonlinear, wieber2006trajectory}, aerospace~\citep{eren2017model, di2018real, malyuta2021advances}, and process systems~\citep{qin2003survey}. One of the main strengths of \ac{MPC} is that it reformulates the optimal control problem as optimization over an open-loop control sequence that is applied successively online. This allows for an implicit form of feedback since the controls are reoptimized from the current state of the system at every time step of the problem~\citep{rawlings2017model}. 

However, when applied to autonomous systems acting in the real-world, deterministic \ac{MPC} is often unable to respond to large disturbances that occur due to environmental factors, model uncertainty, etc. Additionally, under such large disturbances the open-loop optimal control may be infeasible or result in unsafe solutions, which, e.g., crash into obstacles. This motivates the development of control algorithms that explicitly account for unknown disturbances in the dynamics and guarantee robustness. This field of study is known as \textit{robust control}~\citep{safonov2012origins,petersen2014robust}, for which two classes of algorithms exist. The first formulates the problem as a min-max optimization, finding a control policy that minimizes the cost under worst-case disturbances. These methods are generally impractical as they require optimization over a class of infinite-dimensional control policies and are often too conservative in practice~\citep{mayne2014model,rawlings2017model}.

\begin{figure}[h]
\centering
\vspace{-0.6cm}
\includegraphics[width=0.9\linewidth]{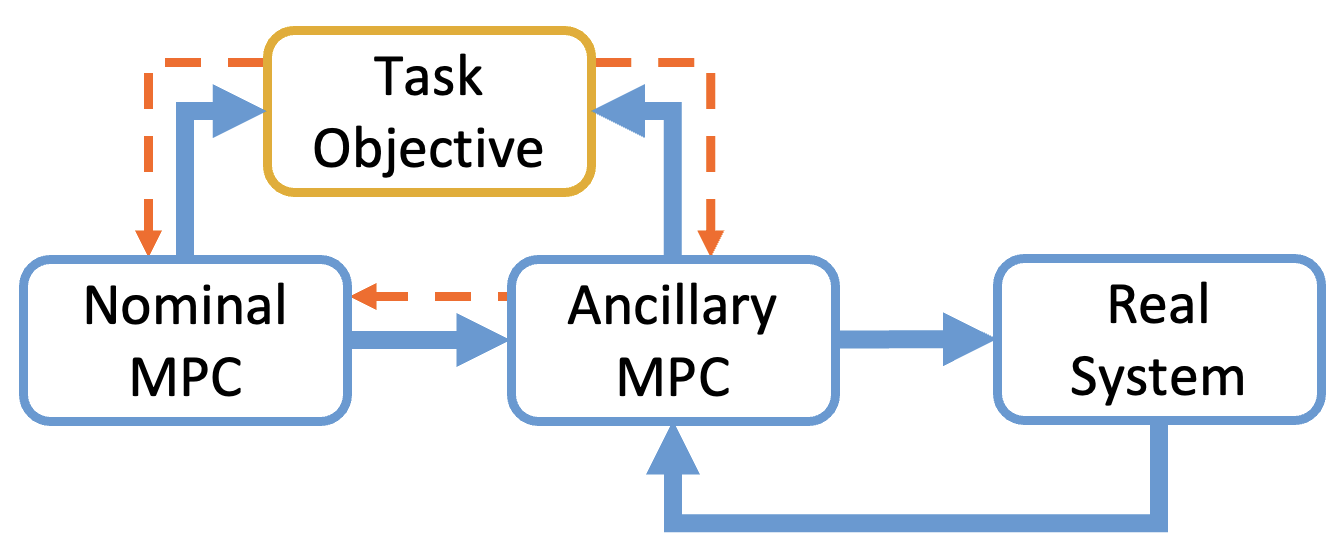}
\caption{Proposed differentiable robust MPC architecture. Orange dashed arrows show how gradients are passed in our architecture.}
\label{fig:architecture}
\end{figure}

\begin{figure}[h]
\centering
\vspace{-0.5cm}
\includegraphics[width=0.75\linewidth]{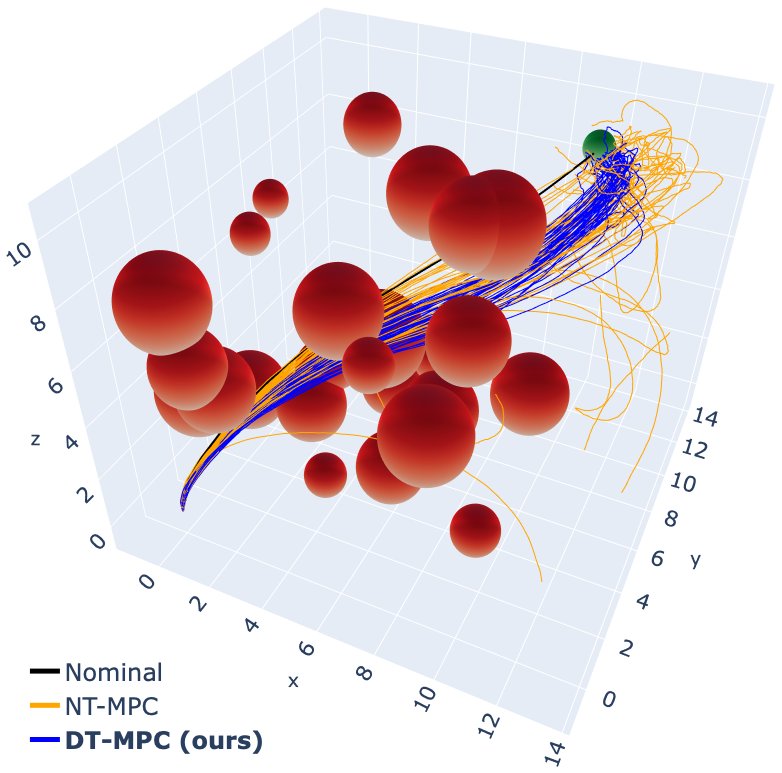}
\caption{Controlled quadrotor trajectories subject to large disturbances. 50 trajectories are plotted for each algorithm. `Nominal' corresponds to the reference trajectory being tracked by the two algorithms. Our proposed differentiable tube-based MPC (DT-MPC) is safer and more robust than the baseline nonlinear tube-based MPC (NT-MPC).}
\label{fig:quadrotor_comparison}
\vspace{-2mm}
\end{figure}

On the other hand, tube-based approaches have found large success as they are able to leverage the strengths of \ac{MPC}~\cite{limon2010robust, mayne2011tube,buckner2018tube,williams2018robust,mammarella2019attitude}. These approaches are robust to uncertainty in the dynamics by dividing the online control problem into two optimization layers: a nominal MPC layer and an ancillary MPC layer. The nominal MPC generates a nominal reference trajectory $\bar{\boldsymbol{x}}$ in the absence of noise or disturbances. The role of the ancillary MPC is to track the reference trajectory subject to uncertainty in the state of the system. Under certain regularity assumptions, the true state $x$ is guaranteed to lie within some bounded distance from the nominal state $\bar{x}$ --- plotting multiple realizations of the state will look like a ``tube'' centered around the nominal trajectory, from which these approaches derive their name~\citep{mayne2011robust,rawlings2017model}. The use of the ancillary controller makes tube-based approaches more flexible than standard open-loop control through the additional degree of freedom available~\citep{mayne2011tube}, while the use of MPC keeps these approaches computationally efficient, particularly for real-time applications, e.g., vehicle trajectory planning. While numerous approaches exist that allow the nominal controller of the tube-based MPC to respond to the environment, they are often heuristic in nature or require running an additional optimization, e.g., to select the best candidate starting state for the nominal MPC at every time step~\cite{mayne2011tube}. 


The main contribution of this work is the development of a novel differentiable tube-based MPC (DT-MPC) framework for safe, robust control. Safety is enforced through the use of discrete barrier states~\citep{almubarak2022safety}, which enables scalable constraint satisfaction such that safe planning and control can be executed in real-time. A general \emph{differentiable optimal control} algorithm is derived that describes how to efficiently pass derivatives through both layers of the MPC architecture (nominal and ancillary controllers). \cref{fig:architecture} summarizes the proposed differentiable robust MPC architecture and \cref{fig:quadrotor_comparison} demonstrates the proposed DT-MPC algorithm solving a complex quadrotor navigation task in the presence of numerous obstacles. Trajectories controlled under DT-MPC remain safe under large uncertainty and arrive at the goal with higher probability than the baseline.

Our proposed approach performs gradient updates on the controller parameters in a principled manner and has computational complexity equivalent to that of a single finite-horizon LQR solve. The computation scales linearly with the look-ahead horizon of the MPC, and, therefore, our proposed algorithm is comparable in complexity to conventional MPC. The efficacy of the proposed approach is benchmarked through multiple Monte Carlo simulations on five nonlinear robotics systems, including two systems in the MuJoCo simulator environment~\citep{todorov2012mujoco}.
A hardware experiment on the Robotarium \cite{wilson2020robotarium} demonstrates the ability of the proposed DT-MPC to adapt to an out-of-distribution test case. Moreover, timing and gradient error comparisons against state-of-the-art differentiable optimal control methods are provided.

In summary, the main contributions of this work include:
\begin{enumerate}
    \item the derivation of a general differentiable optimal control framework enabled through a novel application of the implicit function theorem,
    \item the proposition of a differentiable tube-based MPC algorithm that allows the online adaptation of controller parameters to maximize success and safety,
    \item extensive benchmarks on multiple nonlinear robotics systems both in simulation and on hardware showing the applicability and generality of the proposed approach.
\end{enumerate}

In \cref{sec:tube_mpc}, tube-based MPC is reviewed along with embedded discrete barrier states, which will be used in the proposed algorithm to enforce safety. \Cref{sec:doc} presents a generalized differentiable optimal control framework, relating recent advancements in differentiable optimization-based control. In \cref{sec:dt_mpc}, a differentiable robust MPC algorithm is presented. Experiments are provided on various robotics systems in \cref{sec:experiments}. Concluding remarks are given in \cref{sec:conclusion}.

\section{Mathematical Background}
\label{sec:tube_mpc}

\subsection{Tube-based Model Predictive Control}
\label{subsec:tube_mpc}

Tube-based MPC approaches design a robust controller for the uncertain, safety-critical system
\begin{align}
    x_{t + 1} = f_\text{true}(x_t, u_t) = f(x_t, u_t) + w_t,  \label{eq:full_dynamics}
\end{align}
where $t \in \mathbb{N}$ denotes the task time step, $f_\text{true}$ describes the true dynamics of the system (e.g., reality), and $f$ is a smooth function which is a model or an approximation of~$f_\text{true}$ (e.g., physics-based or learned dynamics) that will be used for MPC. $w_t \in \mathbb{W}$, where $\mathbb{W} \subseteq \mathbb{R}^{n_w}$ is a convex and compact set containing the origin, is a bounded disturbance, which exists due to model uncertainty, random noise, etc.~\citep{mayne2011tube}. The system is subject to control constraints $u_t \in \mathbb{U} \subset \mathbb{R}^{n_u}$, such as physical actuator limits, and safety constraints $x_t \in \mathbb{X} \subset \mathbb{R}^{n_x}$ that determine the safe operating region of the system of interest. It is assumed that both sets are (strict) superlevel sets of some continuously differentiable functions.

Development of a successful model predictive controller for \cref{eq:full_dynamics} relies on the fact that $f$ is a good approximation of the true dynamics~$f_\text{true}$. However, in practice, the system under study is subject to large dynamical uncertainty through effects such as unmodeled physics, random noise, etc., that results in some error between $f_\text{true}$ and $f$, which is captured by the additive disturbance $w$ in the formulation above. The objective will be to bound the deviation of the true state trajectory $\boldsymbol{x} \defeq ( x_0, x_1, \ldots )$, subject to disturbances, from a reference trajectory $\bar{\boldsymbol{x}} \defeq ( \bar{x}_0, \bar{x}_1, \ldots )$ that is determined online by solving the nominal MPC problem defined below. Throughout this work, we will use bold letters, such as $\boldsymbol{x}$, to represent trajectories of variables. For notational compactness, we define the state-control trajectory as $\boldsymbol{\tau} \defeq (\boldsymbol{x}, \boldsymbol{u}) = (x_0, u_0, x_1, u_1, \ldots, x_N)$, with planning horizon $N \in \mathbb{N}$. The nominal MPC problem is, therefore, given as:
\begin{problem}[Nominal MPC] \label{prob:nominal_mpc}
\begin{align*}
\bar{\boldsymbol{\tau}} = & \argmin_{\boldsymbol{\tau}} \bar{J}(\boldsymbol{\tau}) \defeq \sum_{k = 0}^{N - 1} \bar{\ell}(x_k, u_k) + \bar{\phi}(x_N), \\
\text{subject to} \ \ & x_{k + 1} = f(x_k, u_k), \quad x_0 = \bar{x}_t, \\ 
& u_k \in \mathbb{U}, \quad x_k \in \bar{\mathbb{X}} \subset \mathbb{X}, 
\end{align*}
where $\bar{x}_t$ denotes the nominal state at MPC time step $t$ since the controller is employed in a receding horizon fashion and re-optimized at every successor state and $k$ is used to denote predictive quantities. $\bar{\ell}:~\mathbb{R}^{n_x}~\times~\mathbb{R}^{n_u}~\to~\mathbb{R}$ and $\bar{\phi}:~\mathbb{R}^{n_x}~\to~\mathbb{R}$ are the nominal running and terminal cost functions, respectively, which determine the task to be solved.
The set $\bar{\mathbb{X}}$ represents a ``tightened'' constraint set that ensures the true state can be kept safe in the presence of uncertainty~\citep{mayne2011tube,rawlings2017model}. In other words, the nominal state must be kept sufficiently far from the boundary of the true constraint set $\mathbb{X}$, such that the true state of the system will not violate the safety constraints when subject to disturbances.
\end{problem}

The nominal controller is often employed in isolation for the control of \cref{eq:full_dynamics} by solving \cref{prob:nominal_mpc} online from every true state $x_t$. This approach is inherently robust to small uncertainty providing one explanation for the success of nominal MPC in practice, even when \cref{prob:nominal_mpc} is not solved fully at every time step due to restrictions on the available computational resources~\cite{pannocchia2011conditions,allan2017inherent}.
However, this approach is only valid when the error between $f_\text{true}$ and $f$ is small, i.e., when the disturbances are small. A potential failure mode of nominal MPC when applied for the control of the true system is safety violations caused by this large predictive error \cite{mayne2011tube}.

To address this shortcoming, tube-based MPC augments the nominal controller with a feedback model predictive controller that drives the state of the true system towards the nominal trajectory. This controller is known as the \emph{ancillary MPC}, and solves online at every true state $x_t$ the following optimization:
\begin{problem}[Ancillary MPC] \label{prob:ancillary_mpc}
\begin{gather*}
\begin{aligned}
\boldsymbol{\tau}^* = \argmin_{\boldsymbol{\tau}} J(\boldsymbol{\tau}, t) \defeq {}&\sum_{k = 0}^{N - 1} \ell(x_k, u_k, \bar{x}_k, \bar{u}_k) \\
&\quad + \phi(x_N, \bar{x}_N),
\end{aligned} \\
\text{subject to} \ \  x_{k + 1} = f(x_k, u_k), \quad x_0 = x_t, \\ 
 u_k \in \mathbb{U}, \quad x_k \in \mathbb{X}, 
\end{gather*}
where $\ell: \mathbb{R}^{n_x} \times \mathbb{R}^{n_u} \times \mathbb{R}^{n_x} \times \mathbb{R}^{n_u} \to \mathbb{R}$ and $\phi: \mathbb{R}^{n_x} \times \mathbb{R}^{n_x} \to \mathbb{R}$ are the ancillary running and terminal cost functions, respectively, often chosen as $\ell = \norm{x - \bar{x}}_{Q}^2 + \norm{u - \bar{u}}_{R}^2$ and $\phi = \norm{x - \bar{x}}_{Q}^2$ for some positive definite matrices $Q$ and $R$. In \cref{sec:dt_mpc}, differentiable optimization will be applied to make $Q$ and $R$ learnable, allowing for online adaptation and tuning of the ancillary MPC.
Due to the disturbances entering the system~\cref{eq:full_dynamics}, the true state~$x_t \neq \bar{x}_t$ in general. The role of \cref{prob:ancillary_mpc} is to drive the true state of the system towards the reference trajectory~$\bar{\boldsymbol{\tau}}$, under the nominal predictive model $f$. The disturbances affect the optimization through perturbations on the initial state of~\cref{prob:ancillary_mpc}. For more details on the tube-based MPC approach and its analyses, the reader is referred to the works of \citet{mayne2011tube} and \citet{rawlings2017model}.
\end{problem}


\subsection{Embedded Barrier States} \label{subsec: dbas}
Consider the safe set $\mathbb{X}$ that is defined as the strict superlevel set of a continuously differentiable function $h: \mathbb{R}^{n_x} \to \R$ such that
\begin{align*}
    \mathbb{X} \defeq \{ x \in \mathbb{R}^{n_x} \ | \ h(x) > 0 \} .
\end{align*}
The goal is to render the safe set $\mathbb{X}$ forward invariant, i.e., once the system is in the set, it stays in it for all future times of operation. This is accomplished by defining an appropriate barrier function $B: \mathbb{X} \to \mathbb{R}$ over the safety condition $h$ whose value ``blows up'' as the state of the system approaches the unsafe region. The idea of \textit{embedded barrier states} is to augment the system with the state of the barrier and define a new control problem that, when solved, guarantees safety along with other performance objectives~\citep{almubarak2022safety,almubarak2021safety}. The \ac{DBaS}, denoted $b_k$, is defined by the dynamics
\begin{align*}
    b_{k + 1} = g(x_k, u_k, b_k) = B(h(f(x_k, u_k)) - \gamma (B(h(x_k)) - b_k) ,
\end{align*}
where $\gamma \in [-1, 1]$.
Examples of suitable barrier functions include the inverse barrier $B(\zeta) = 1 / \zeta$ and the log barrier $B(\zeta) = -\log \zeta$. The state of the system to be controlled is then augmented by the barrier state $\hat{x}_k = (x_k, b_k)$, resulting in the augmented dynamics
\begin{align}
    \hat{x}_{k + 1} = \hat{f}(\hat{x}_k, u_k) = \begin{bmatrix} f(x_k, u_k) \\ g(x_k, u_k, b_k) \end{bmatrix} . \label{eq:safety_embedded_sys}
\end{align}
The system \cref{eq:safety_embedded_sys} is called the \emph{safety-embedded system}. The set $\mathbb{X}$ is, therefore, rendered forward invariant if and only if the barrier state of the system remains bounded for all time, given that the system starts in the safe set ~\citep{almubarak2022safety}.

Using the \ac{iLQR}~\citep{li2004iterative} with embedded DBaS guarantees positive definiteness of the Hessians of the value function, providing guarantees of improvement of safe solutions and convergence (see Theorems 2 and 3 of \cite{almubarak2022safety}). The algorithm enjoys a computational efficiency and fast convergence as shown by \citet{almubarak2022safety} and in the MPC formulation by \citet{cho2023model} compared to other safe DDP-based approaches, such as penalty methods, control barrier functions, and augmented Lagrangian methods, which require multiple optimization loops and thus do not scale well for real-time applications.

Nonetheless, in the interest of the proposed work and in light of the requirement of starting in the safe region, the solution to the open-loop optimal control problem may be infeasible. A well established solution for barrier-based methods is the \textit{relaxed} barrier function, which was presented in the linear MPC formulation by \citet{feller2016relaxed} using the logarithmic barrier function. In dependence of the underlying relaxation, convergence guarantees and constraint satisfaction were provided. Hence, in this work, recursive feasibility of the tube-based MPC is pledged through the use of a \emph{relaxed embedded barrier state} that is similar conceptually to the idea of relaxed barrier functions~\citep{feller2016relaxed}. In essence, the barrier function~$B$ is replaced with the \textit{relaxed} barrier function by taking a Taylor series approximation of the safety function~$h$ around a point $\alpha$ close to the constraint ``unsatisfaction'', i.e., close to~$h=0$, while reserving the advantages of using barrier states within the iLQR. For example, for the inverse barrier function, which is used throughout this work, we take a quadratic approximation and define the strictly monotone and continuously differentiable \textit{relaxed} barrier function:
\begin{align*}
B_\alpha(\zeta) = \begin{cases}
    1 / \zeta & \text{ if } \zeta \geq \alpha ,\\
    1 / \alpha - (\zeta - \alpha)/\alpha^2 + (\zeta - \alpha)^2 / \alpha^3 & \text{ if } \zeta < \alpha .
\end{cases}
\end{align*}
It is worth noting that the relaxation hyperparameter $\alpha$ determines the amount of relaxation such that $\lim_{\alpha \to 0} B_{\alpha}(\zeta) \to B(\zeta)$. Enabled through differentiable optimization, the framework described in \cref{sec:dt_mpc} allows an adaptation scheme to be derived that modulates the amount of relaxation online through updating $\alpha$ based on task performance and feasibility. From the safe MPC via barrier methods perspective, the proposed work provides a novel expansion of the works \cite{almubarak2022safety}, \cite{cho2023model} and \cite{feller2016relaxed} to a tube-based MPC algorithm in which the barrier state's parameters, e.g., dynamics and cost penalization, can be adapted or auto-tuned which greatly improves the controls performance, as we show and discuss in more detail in our experimentation provided in \cref{sec:experiments}.

In the sequel, we will use $x$ and $f$ to represent the embedded state and the safety-embedded dynamics in our development of the main proposition for notational simplicity. It should be understood that we are working with the safety-embedded system and safety is guaranteed by the boundedness of the barrier state.


\section{Generalized Differentiable Optimal Control through the Implicit Function Theorem}
\label{sec:doc}


We derive a general differentiable optimal control framework for computing gradients through the solution of a parameterized optimal control problem ($\boldsymbol{\tau}^*$ of \cref{prob:parameterized_oc} below). This enables a principled manner through which the parameters of the control problem can be automatically adapted to maximize both safety as well as the desired task performance. The key result here is that the differentiable optimal control algorithm has the same computational complexity as a single finite-horizon LQR solve, namely $O(N)$ in time. This connection with LQR control justifies the use of a Gauss-Newton approximation when backpropagating the derivatives through time, allowing the dynamics and cost derivatives to be reused between the optimal control solver (e.g., iLQR) and the differentiable optimization. Furthermore, we show how the gradient can be accumulated across the trajectory at an $O(1)$ memory cost, independent of the time horizon. This prevents needing to store the entire trajectory of derivatives in memory at once, which is an important consideration for highly parameterized problems, e.g., when neural networks are used to model the dynamics or cost function.

\subsection{Differentiable Optimization}
\label{subsec:diff_opt}

Our work is inspired by recent developments in implicit differentiation --- also known as \emph{differentiable optimization} --- which is an emerging trend in machine learning~\citep{amos2017optnet,bai2019deep,bolte2021nonsmooth,blondel2022efficient} that studies how to embed optimization processes as end-to-end trainable components into learning-based architectures.
Differentiable optimization has seen large success in a wealth of fields such as hyperparameter tuning~\citep{bertrand2020implicit,lorraine2020optimizing}, meta-learning~\citep{franceschi2018bilevel,rajeswaran2019meta}, and model-based reinforcement learning and control~\citep{amos2018differentiable,jin2020pontryagin,dinev2022differentiable,grandia2023doc}. This subsection gives a brief overview of differentiable optimization in the context of general optimization problems. These results will be used afterwards to develop a general differentiable optimal control methodology and a differentiable robust MPC framework.


Consider the unconstrained minimization problem
\begin{align} \label{eq:min_prob}
    z^*(\theta) = \argmin_z \varphi(z, \theta),
\end{align}
where $\varphi: \R^{n_z} \times \R^{n_\theta} \to \R$ is twice continuously differentiable, $z$ is the optimization variable, and $\theta$ is a vector of parameters that influence the objective function to be minimized. The \ac{IFT} provides a precise relation between the optimal solution and the learning parameters by defining $z^*$ as an implicit function of $\theta$. Furthermore, it gives an expression for the derivative $\pdv{z^*}{\theta}$ --- the Jacobian of the solution with respect to the parameters --- and the conditions under which this derivative is defined. This development is necessary when the minimization process \cref{eq:min_prob} is embedded as a component within, e.g., a deep learning architecture. This allows the parameters of the optimization process to be optimized in an end-to-end fashion through the use of backpropagation and other gradient-based algorithms.
\begin{theorem}[Implicit function theorem~\citep{krantz2002implicit}]\label{thm:ift}
Let $F: \R^{n_z} \times \R^{n_\theta} \to \R^{n_z}$ be a continuously differentiable function. Fix a point $(z_0, \theta_0)$ such that $F(z_0, \theta_0) = 0$. If the Jacobian matrix of partial derivatives $\pdv{F}{z}(z_0, \theta_0)$ is invertible, then there exists a function $z^*(\cdot)$ defined in a neighborhood of $\theta_0$ such that $z^*(\theta_0) = z_0$ and
\begin{align*}
    \pdv{}{\theta}z^*(\theta) = - \left(\pdv{}{z}F(z^*(\theta), \theta)\right)^{-1} \pdv{}{\theta}F(z^*(\theta), \theta) .
\end{align*}
\end{theorem}
\begin{proof}
See \citet{krantz2002implicit} or \citet{de2012implicit}.
\end{proof}

\begin{remark}
The condition $F(z_0, \theta_0) = 0$ may seem restrictive, but the power of \cref{thm:ift} is revealed when we examine the first-order optimality conditions for the minimization problem \cref{eq:min_prob} --- a solution $z^*$ to \cref{eq:min_prob} must satisfy $\nabla_z \varphi(z^*, \theta) = 0$. Therefore, the gradient of the function $\varphi$ satisfies the conditions of \cref{thm:ift} as long as the matrix of second-order partial derivatives of $\varphi$ with respect to $z$, namely the Hessian $\varphi_{zz} \defeq \pdv{}{z} \nabla_z \varphi$ is invertible. Furthermore, this generalizes naturally to constrained optimization problems (e.g., $z \in \mathcal{C} \subset \mathbb{R}^{n_z}$, for some constraint set $\mathcal{C}$) by considering the appropriate first-order optimality conditions of the problem (e.g., the \ac{KKT} conditions). Further discussion can be found in recent works such as \cite{blondel2022efficient}.
\end{remark}

The \ac{IFT} will be applied in the following subsection to derive a general \ac{DOC} framework.
This framework will be used to develop a robust MPC algorithm in \cref{sec:dt_mpc} that is made adaptive through online differentiable optimization.

\subsection{General Learning Framework as Bilevel Optimization}

We start by motivating the learning problem through the lens of optimal control. For clarity of presentation, we focus on the unconstrained case and provide discussion on the control-constrained case in \cref{appendix:control_constraints}. We consider the following general \emph{parameterized} optimal control problem of the form
\begin{problem}[Parameterized OC] \label{prob:parameterized_oc}
\begin{align*}
\boldsymbol{\tau}^*(\theta) = \argmin_{\boldsymbol{\tau}} {}& J(\boldsymbol{\tau}, \theta) \defeq \sum_{k = 0}^{N - 1} \ell(x_k, u_k, \theta) + \phi(x_N, \theta), \\
\text{subject to} \ \ & x_{k + 1} = f(x_k, u_k, \theta), \quad x_0 = \xi(\theta),
\end{align*}
where $\xi: \mathbb{R}^{n_\theta} \to \mathbb{R}^{n_x}$ is a differentiable function denoting the initial condition of the problem (e.g., $\bar{x}_t$ for the nominal MPC and $x_t$ for the ancillary MPC), $f: \mathbb{R}^{n_x} \times \mathbb{R}^{n_u} \times \mathbb{R}^{n_\theta} \to \mathbb{R}^{n_x}$ is the parameterized dynamics, and $\ell: \mathbb{R}^{n_x} \times \mathbb{R}^{n_u} \times \mathbb{R}^{n_\theta} \to \mathbb{R}$ and $\phi: \mathbb{R}^{n_x} \times \mathbb{R}^{n_\theta} \to \mathbb{R}$ are the parameterized running and terminal cost functions, respectively.
\end{problem}

The solution $\boldsymbol{\tau}^*$ depends on the problem parameters $\theta$, which represent the parts of the dynamics and the objective that are learnable or adaptable, such as cost function weights or unknown constants of a parametric, physics-based model. In practice, these parameters $\theta$ are hand-tuned by a domain expert and fixed during task execution. However, we propose an alternative methodology enabled through differentiable optimization that allows the parameters to be learned and adapted online through minimization of an appropriately specified loss function describing the desired task behavior. The learning objective is therefore defined as the following bilevel optimization over the parameters of \cref{prob:parameterized_oc}:
\begin{problem}[Learning Problem] \label{prob:learning_prob}
\begin{align}
    \min_\theta L(\boldsymbol{\tau}^*(\theta)) , \label{eq:learning_obj}
\end{align}
where $L$ is a differentiable loss function, such as the imitation loss $L = \norm{\boldsymbol{\tau}^*(\theta) - \boldsymbol{\tau}_\text{expert}}^2$ where $\boldsymbol{\tau}_\text{expert}$ is generated by a task expert, e.g., through human demonstration. Without loss of generality, we will assume $L$ does not depend on $\theta$ directly, but the following results can be easily extended to such cases.
\end{problem}

The goal will be to establish an efficient methodology to learn the optimal parameters $\theta^*$ by solving \cref{prob:learning_prob}. A natural choice to learn these parameters is through gradient descent. Using the chain rule, the gradient of the objective \cref{eq:learning_obj} with respect to $\theta$ is given as
\begin{align*}
    \nabla_{\theta} L(\boldsymbol{\tau}^*(\theta)) = \left(\pdv{\boldsymbol{\tau}^*(\theta)}{\theta}\right)^\top \nabla_{\boldsymbol{\tau}} L(\boldsymbol{\tau}^*(\theta)) .
\end{align*}
This quantity is often referred to as a \emph{hypergradient} to distinguish it from the gradients of the lower-level problem.
The term $\nabla_{\boldsymbol{\tau}} L$ can be calculated straightforwardly as it is often a simple analytic expression, e.g., for the imitation learning loss defined earlier, the gradient is given simply as $\nabla_{\boldsymbol{\tau}} L = 2 (\boldsymbol{\tau} - \boldsymbol{\tau}_\text{expert})$. The difficulty arises in calculating the Jacobian $\pdv{\boldsymbol{\tau}^*}{\theta}$ efficiently --- it is not obvious at first how to take derivatives of a \emph{solution} to \cref{prob:parameterized_oc}.

A na\"ive approach to computing this Jacobian would be to ``unroll'' the optimization algorithm itself, in a process similar to \ac{AD}. Indeed, this approach is used to great effect in recent work~\citep{okada2017path,bhardwaj2020differentiable}. However, the computational efficiency scales linearly with the number of iterations $K$ necessary to solve \cref{prob:parameterized_oc}, making it expensive for highly nonlinear, non-convex problems that might require many iterations to find a good solution. This motivates the use of implicit differentiation and the IFT as introduced in \cref{subsec:diff_opt}, which enables us to derive an analytic expression for the Jacobian $\pdv{\boldsymbol{\tau}^*}{\theta}$ without requiring full unrolling of the lower-level optimizer.

We begin by introducing the optimality conditions for \cref{prob:parameterized_oc}. This is a well-known result in optimal control known as Pontryagin's maximum principle \citep{pontryagin2018mathematical}.
\begin{proposition}[Optimality conditions of \cref{prob:parameterized_oc}]\label{proposition:optimality_conds}
Define the real-valued function $\mathcal{L}$ for \cref{prob:parameterized_oc}, called the \emph{Lagrangian}, as
\begin{align*}
\begin{aligned}
    \mathcal{L}(\boldsymbol{z}, \theta) = {}& \sum_{k = 0}^{N - 1} \ell(x_k, u_k, \theta) + \lambda_{k + 1}^\top (f(x_k, u_k, \theta) - x_{k + 1}) \\
    & + \lambda_0^\top (\xi(\theta) - x_0) + \phi(x_N, \theta),
\end{aligned}
\end{align*}
where $\lambda_k \in \mathbb{R}^{n_x}$, $k = 0, 1, \ldots, N$ are the Lagrange multipliers for the dynamics and initial condition constraints, and $\boldsymbol{z} \defeq (\boldsymbol{\tau}, \boldsymbol{\lambda}) = (\lambda_0, x_0, u_0, \ldots, \lambda_N, x_N)$ for notational compactness.

Let $\boldsymbol{\tau}^*$ be a solution to \cref{prob:parameterized_oc} for fixed parameters~$\theta$. Then, there exists Lagrange multipliers~$\boldsymbol{\lambda}^*$ which together with $\boldsymbol{\tau}^*$ satisfy $\nabla_{\boldsymbol{z}} \mathcal{L}(\boldsymbol{z}^*, \theta) = 0$.
\end{proposition}
\begin{proof}
The conditions $\nabla_{\boldsymbol{z}} \mathcal{L}(\boldsymbol{z}^*, \theta) = 0$ are the \ac{KKT} conditions for \cref{prob:parameterized_oc}. See \cref{appendix:optimality_conditions}.
\end{proof}

For reasons that will be clear soon, it is advantageous to use \ac{DDP}~\citep{mayne1966second,jacobson1970differential} as the optimization algorithm to solve \cref{prob:parameterized_oc}. \ac{DDP} uses the principle of dynamic programming~\citep{bellman1966dynamic} to solve \cref{prob:parameterized_oc} efficiently by reparameterizing the minimization over all possible control trajectories as a \emph{sequence} of minimizations proceeding backwards-in-time.

More formally, \ac{DDP} iteratively solves for and applies the Newton step $\odif{\boldsymbol{z}}$ until convergence:
\begin{align}
    \mathcal{L}_{\boldsymbol{z}\boldsymbol{z}} \odif{\boldsymbol{z}} = - \nabla_{\boldsymbol{z}} \mathcal{L}, \label{eq:ddp_newton_step}
\end{align}
which requires inverting the Hessian $\mathcal{L}_{\boldsymbol{z}\boldsymbol{z}}$.
However, na\"ively inverting this matrix is prohibitively expensive, since the size of the matrix is quadratic in the time horizon of the problem, i.e., the matrix inversion of $\mathcal{L}_{\boldsymbol{z}\boldsymbol{z}}$ is $O(N^3)$. To remedy this, we utilize dynamic programming and the sparsity induced by the constraints of the control problem to rewrite \cref{eq:ddp_newton_step} as a series of backwards difference equations --- in optimal control, these are known as the \emph{Riccati equations}. This allows the solution of \cref{eq:ddp_newton_step} to be computed with efficiency linear in the horizon~$N$. This same insight will be key to efficiently computing the implicit derivative of \cref{prob:parameterized_oc} defined below.

\begin{proposition}[Implicit derivative of \cref{prob:parameterized_oc}]\label{proposition:implicit_derivative}
Let $\boldsymbol{\tau}^*$ be a solution to \cref{prob:parameterized_oc} for fixed parameters~$\theta$. By \cref{proposition:optimality_conds}, \cref{thm:ift} holds with $F = \nabla_{\boldsymbol{z}} \mathcal{L}$. Furthermore, the Jacobian $\pdv{\boldsymbol{z}^*}{\theta}$ is given as
\begin{align}
\pdv{}{\theta}\boldsymbol{z}^*(\theta) = - \mathcal{L}_{\boldsymbol{z}\boldsymbol{z}}^{-1} \mathcal{L}_{\boldsymbol{z}\theta} . \label{eq:implicit_derivative}
\end{align}
\end{proposition}
\begin{proof}
See \cref{appendix:implicit_derivative}.
\end{proof}

\begin{remark}
Note that the structure of both \cref{eq:ddp_newton_step} and \cref{eq:implicit_derivative} here is similar, with the only difference being the fact that the inverse Hessian is applied to the gradient $\nabla_{\boldsymbol{z}} \mathcal{L}$ in \cref{eq:ddp_newton_step} while it is multiplied with a matrix of partial derivatives $\mathcal{L}_{\boldsymbol{z}\theta}$ in \cref{eq:implicit_derivative}. In fact, solving \cref{eq:implicit_derivative} directly allows one to derive the \ac{PDP} framework proposed by \citet{jin2020pontryagin}, which was originally derived by directly differentiating the KKT conditions of \cref{prob:parameterized_oc} with respect to $\theta$.
\end{remark}
\begin{corollary}[Pontryagin differentiable programming~\citep{jin2020pontryagin}]\label{corollary:pdp}
Let the conditions of \cref{proposition:optimality_conds} and \cref{proposition:implicit_derivative} hold. Then, the differentiable Pontryagin conditions ((Eq. (13) of \citet{jin2020pontryagin}) are equivalent to solving the linear system \cref{eq:implicit_derivative}.
\end{corollary}
\begin{proof}
This can be seen by expanding the matrix multiplication in \cref{eq:implicit_derivative}. See \cref{appendix:pdp}.
\end{proof}
\cref{corollary:pdp} implies that rather than differentiating the Pontryagin conditions directly, a more general framework can be derived by applying the IFT to the control problem. Furthermore, the IFT describes the conditions under which this derivative exists and can be calculated, which generalizes the work by \citet{jin2020pontryagin}.

However, this process requires solving the set of matrix equations in \cref{eq:implicit_derivative}, i.e., a matrix control system backwards-in-time. This is inefficient due to requiring the computation and storage of the intermediate Jacobians $\pdv{x_k^*}{\theta}$ and $\pdv{u_k^*}{\theta}$ along the entire trajectory. We will show next that an improvement can be made by taking advantage of an insight similar to the computation of \acp{VJP} in \ac{AD}.

\begin{theorem}[Differentiable Optimal Control] \label{thm:doc}
Let $\boldsymbol{z}$ denote the augmented vector consisting of $\boldsymbol{\tau}$ and $\boldsymbol{\lambda}$.
In addition, let the conditions of \cref{proposition:optimality_conds} and \cref{proposition:implicit_derivative} hold. Then, the gradient of the loss $L$ with respect to $\theta$ is given by
\begin{align}
    \nabla_\theta L(\boldsymbol{z}^*(\theta)) = \mathcal{L}_{\theta \boldsymbol{z}} \delta \boldsymbol{z} , \label{eq:L_grad_doc}
\end{align}
where the vector $\delta \boldsymbol{z}$ is given by solving the linear system
\begin{align}
    \mathcal{L}_{\boldsymbol{z}\boldsymbol{z}} \delta \boldsymbol{z} = -\nabla_{\boldsymbol{z}} L . \label{eq:d_z}
\end{align}
\end{theorem}
\begin{proof}
See \cref{appendix:doc}.
\end{proof}
\begin{remark}
\Cref{thm:doc} shows that, rather than calculating $\mathcal{L}_{\boldsymbol{z}\boldsymbol{z}}^{-1} \mathcal{L}_{\boldsymbol{z} \theta}$ directly as in \cref{eq:implicit_derivative}, we can instead start by solving the linear system \cref{eq:d_z} for the vector~$\delta \boldsymbol{z}$. Then, the gradient $\nabla_\theta L$ can be computed through a simple matrix multiplication (\cref{eq:L_grad_doc}). This general algorithm is presented in \cref{alg:doc}, whose full derivation is provided in the \cref{appendix:doc_algorithm}.
\end{remark}

\begin{algorithm}
\caption{Differentiable Optimal Control (DOC)} \label{alg:doc}
\KwIn{Derivatives of $\mathcal{L}$ (equivalently $f$, $\ell$, $\phi$, and $\xi$) and $L$ along the solution $\boldsymbol{z}^*$}
\KwOut{Gradient of upper-level loss $\nabla_\theta L$}

$\widetilde{\boldsymbol{V}}_x, \boldsymbol{V}_{xx}, \widetilde{\boldsymbol{k}}, \boldsymbol{K} \gets$ Solve backward pass equations (\cref{alg:doc_bpass} of \cref{appendix:doc_algorithm}); \\
$\nabla_\theta L \gets$ Solve forward pass equations (\cref{alg:doc_fpass} of \cref{appendix:doc_algorithm});
\end{algorithm}

As established earlier and presented in \cref{alg:doc}, solving systems of the form \cref{eq:d_z} can be accomplished through the \ac{DDP} equations by replacing the gradient of the Lagrangian $\nabla_{\boldsymbol{z}} \mathcal{L}$ in \cref{eq:ddp_newton_step} with the gradient of the upper-level loss $\nabla_{\boldsymbol{z}} L$. This fact illustrates the connection between the lower-level control problem (\cref{prob:parameterized_oc}) and the upper-level learning problem (\cref{prob:learning_prob}) and highlights the advantages of implicit differentiation --- by taking advantage of the structure of the lower-level problem, an efficient algorithm can be derived for computing hypergradients of the upper-level problem. This algorithm is presented in \cref{alg:doc} and has $O(N)$ time complexity and $O(1)$ memory complexity, where $N$ is the look-ahead horizon of the MPC. Therefore, our approach scales similarly to conventional MPC, and furthermore, this connection shows that it is advantageous to use DDP as the lower-level control solver since the necessary derivatives for DDP (e.g., $\mathcal{L}_{\boldsymbol{z}\boldsymbol{z}}$) can be reused during the computation of the hypergradient in \cref{eq:L_grad_doc,eq:d_z}.

Remarkably, by deriving \cref{thm:doc} through the \ac{IFT}, the proposed \ac{DOC} methodology is independent of how a solution to \cref{prob:parameterized_oc} is generated. In other words, rather than begin from an existing optimal control algorithm, as is presented by, e.g., \citet{amos2018differentiable} in the context of \ac{iLQR} or~\citet{dinev2022differentiable} in the context of \ac{DDP}, we show that a single framework enables the differentiable optimization of any control algorithm, as long as the produced solution itself satisfies the optimality conditions of \cref{proposition:optimality_conds}.

An additional benefit of the independence from the optimal control solver is the fact that the gradient computation does not depend on the number of solver iterations $K$ required to reach a solution. While unrolling a numerical solver through \ac{AD} would incur computational complexity $O(KN)$ in both time and memory, our approach maintains $O(N)$ complexity independent of the underlying solver. These facts allow the gradient computation to be efficient enough for the adaptation of a real-time MPC controller. Further speedups can be incorporated by adopting a Gauss-Newton approximation when solving \cref{eq:d_z}, at the cost of numerical error. The particular choice of a Gauss-Newton approximation allows one to derive the seminal differentiable MPC (Diff-MPC) by \citet{amos2018differentiable}. This is formalized in the following corollary:
\begin{corollary}[Diff-MPC~\citep{amos2018differentiable}]
Diff-MPC is equivalent to using a Gauss-Newton approximation when solving \cref{eq:d_z}.
\end{corollary}
\begin{proof}
See \cref{appendix:diff_mpc}.
\end{proof}

\subsection{Numerical Precision Guarantees}
\label{subsec:precision}

Recent work by \citet{blondel2022efficient} has provided theoretical guarantees on the numerical precision of implicit differentiation-based approaches for a general class of problems.
In practice, the implicit derivative $\frac{\partial \boldsymbol{z}^*}{\partial \theta}$ is computed at some suboptimal point $\hat{\boldsymbol{z}}$ that approximates the optimal solution $\boldsymbol{z}^*$.
Therefore, it is helpful to understand both theoretically and empirically the error in the Jacobian approximation as it directly affects the quality of the final hypergradients.

Using \cref{proposition:implicit_derivative}, we can define the \emph{Jacobian estimate at $(\hat{\boldsymbol{z}}, \theta)$} as the function $J(\hat{\boldsymbol{z}}, \theta) \defeq - \mathcal{L}_{\boldsymbol{z}\boldsymbol{z}}^{-1}(\hat{\boldsymbol{z}}, \theta) \mathcal{L}_{\boldsymbol{z}\theta}(\hat{\boldsymbol{z}}, \theta) \approx \frac{\partial \boldsymbol{z}^*}{\partial \theta}$.
By assuming $\mathcal{L}_{\boldsymbol{z}\boldsymbol{z}}$ is well-conditioned and Lipschitz and $\mathcal{L}_{\boldsymbol{z}\theta}$ is bounded and Lipschitz, it can be shown that the Jacobian estimate error is on the same order as that of approximating $\boldsymbol{z}^*$ with $\hat{\boldsymbol{z}}$, namely $\norm{J(\hat{\boldsymbol{z}}, \theta) - \frac{\partial \boldsymbol{z}^*}{\partial \theta}} \leq C \norm{\hat{\boldsymbol{z}} - \boldsymbol{z}^*}$ (\cite[Theorem 1]{blondel2022efficient}).

These bounds are not very useful in the context of optimal control since we rarely construct the entire Hessian matrices of the Lagrangian $\mathcal{L}_{\boldsymbol{z}\boldsymbol{z}}$, etc. Furthermore, as shown previously, these matrices have sparse structure due to the time-varying nature of the problem, and thus the bounds for the general case are not very informative. Therefore, we specialize these numerical guarantees for general optimal control problems by showing that the stagewise Jacobian errors, denoted $\norm{ \frac{\partial \hat{x}_k}{\partial \theta} - \frac{\partial x_k^*}{\partial \theta}}$ and $\norm{ \frac{\partial \hat{u}_k}{\partial \theta} - \frac{\partial u_k^*}{\partial \theta}}$, grow linearly with respect to the errors in the solution approximation $\norm{\hat{x}_k - x_k^*}$ and $\norm{\hat{u}_k - u_k^*}$ in a recursive manner. Notably, due to the temporal structure of the control problem, the Jacobian errors at time $k$ only depend on the state and control errors up to and including time $k$. These bounds hold under typical assumptions on quantities related to the optimal control problem, such as the positive definiteness of the matrix $Q_{uu}$ (computed during the backward pass of the algorithm, line 1 of \cref{alg:doc}) and the local Lipschitzness of the dynamics Jacobians. To the authors' best knowledge, this is the first work to present theoretical guarantees on the numerical precision of IFT-based differentiable control algorithms. An abbreviated statement of the theorem is given below with the full version and its proof appearing in \cref{appendix:jac_error}.

\begin{theorem}[Jacobian estimate error]
\label{thm:jac_error}
Let the Jacobians of the dynamics $f_{x_k}, f_{u_k}, f_{\theta_k}$ and the Hessians of the state-action value function $Q_{uu}^{(k)}$, $Q_{ux}^{(k)}$, and $Q_{u\theta}^{(k)}$ be locally Lipschitz in $(x_k, u_k)$ and bounded in a neighborhood of the optimal trajectory $\boldsymbol{z}^*$. Furthermore, let $Q_{uu}^{(k)}$ be positive definite in a neighborhood of $(x_k^*, u_k^*)$ for all $k = 0, 1, \ldots, N - 1$. Then, the error in using \cref{proposition:implicit_derivative} to compute the implicit derivatives of the control problem is upper bounded by
\begin{align*}
    \norm{ \frac{\partial \hat{u}_k}{\partial \theta} - \frac{\partial u_k^*}{\partial \theta}} &\leq \sum_{t = 0}^{k} C_{k,t} (\norm{\hat{x}_t - x_t^*} + \norm{\hat{u}_t - u_t^*}),\\
    \norm{ \frac{\partial \hat{x}_{k + 1}}{\partial \theta} - \frac{\partial x_{k + 1}^*}{\partial \theta}} &\leq \sum_{t = 0}^{k} D_{k+1,t} (\norm{\hat{x}_t - x_t^*} + \norm{\hat{u}_t - u_t^*}) ,
\end{align*}
for constants $C_{k,t}, D_{k+1,t} > 0$.
\end{theorem}

While our proposed algorithm avoids computing the intermediate Jacobian $J(\hat{\boldsymbol{z}}, \theta)$ explicitly, we nevertheless empirically validate the Jacobian estimate error on multiple nonlinear systems to get a quantitative understanding of the numerical precision of our approach in practice.
We run DDP for an increasing number of iterations to generate a sequence of approximate solutions $\hat{\boldsymbol{z}}$ and plot the Jacobian estimate error $\norm{J(\hat{\boldsymbol{z}}, \theta) - \frac{\partial \boldsymbol{z}^*}{\partial \theta}}$ as a function of the iterate error $\norm{\hat{\boldsymbol{z}} - \boldsymbol{z}^*}$. \cref{fig:grad_error} shows the results for the quadrotor system, with plots for the other systems given in \cref{appendix:comparisons}.
Since a closed-form expression of $\frac{\partial \boldsymbol{z}^*}{\partial \theta}$ is not available, we use numerical differentiation (finite differences) to compute the ground truth Jacobian approximately as suggested by \citet{blondel2022efficient}.

\begin{figure}[h]
\centering
\includegraphics[width=0.39\textwidth]{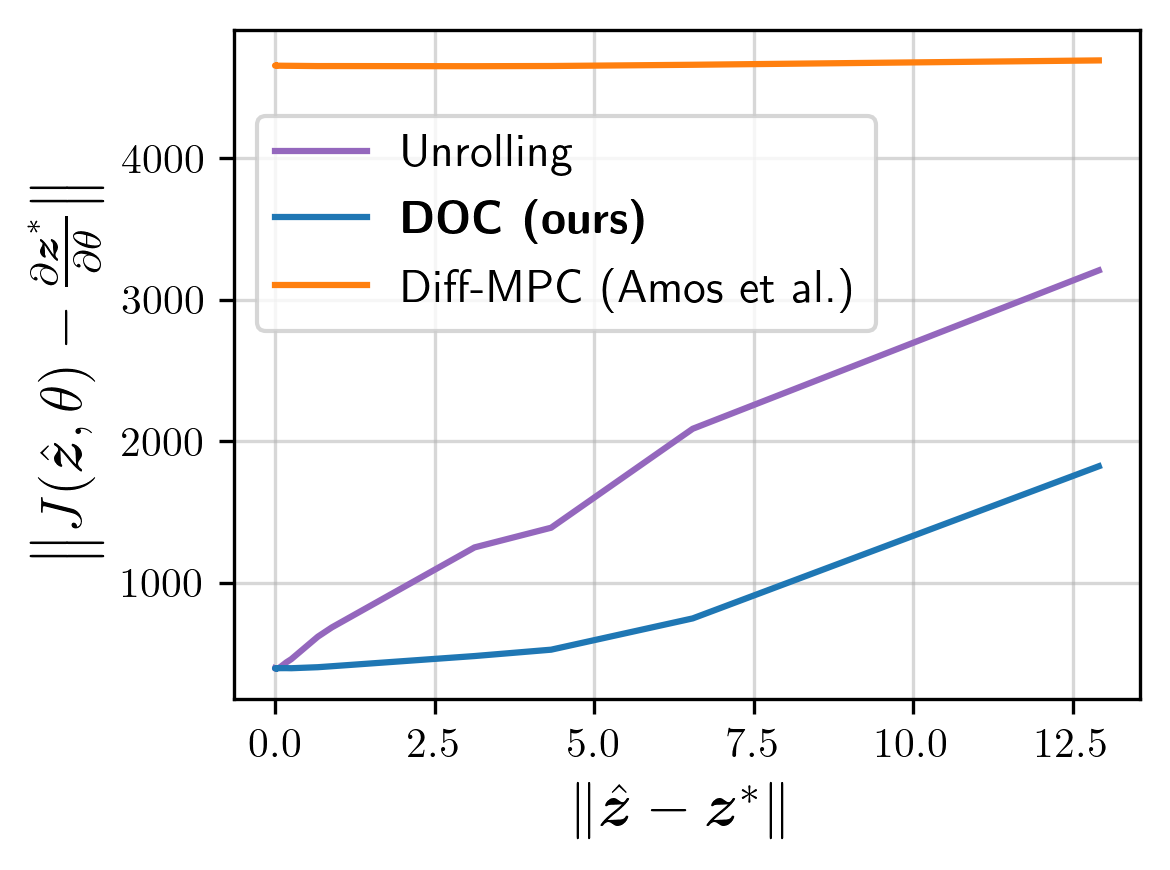}
\caption{Jacobian estimate errors on the quadrotor system as a function of DDP iterate error.}
\label{fig:grad_error}
\end{figure}

``Unrolling'' corresponds to unrolling the DDP iterations as a differentiable compute graph and backpropagating to compute Jacobians via \ac{AD}.
It should be noted there is irreducible error due to the use of finite differences to approximate $\frac{\partial \boldsymbol{z}^*}{\partial \theta}$.
The proposed DOC outperforms both unrolling and Diff-MPC \citep{amos2018differentiable} in terms of numerical error. Diff-MPC has nearly constant error regardless of distance to the optimal solution due to dropping the second-order dynamics derivative terms in the algorithm. In conclusion, we observe a qualitative benefit of our IFT-based approach that supports the theory: there is a large region around the optimal trajectory $\boldsymbol{z}^*$ where the Jacobian error grows very slowly, suggesting our approach is advantageous for producing accurate derivatives in the context of differentiable control.

Further gradient error as well as timing comparisons between our proposed \cref{alg:doc} and the algorithms of \citet{amos2018differentiable}, \citet{jin2020pontryagin}, and \citet{dinev2022differentiable} are provided in \cref{appendix:comparisons}.

\section{Differentiable Tube-based MPC} \label{sec:dt_mpc}



When experiencing large disturbances, MPC under the nominal system model can fail. Although tube-based MPC is designed to tackle such a problem, determining the cost parameters necessary to achieve robustness is difficult in practice due to the nonlinear dependence between the nominal controller and the ancillary controller. Moreover, tuning fixed parameters offline to work for one scenario does not generalize for online, real-time MPC applications, where there may exist, e.g., a large sim-to-real gap. This motivates a method that can update the robust controller parameters automatically and online.

We start by redefining the nominal MPC introduced in \cref{prob:nominal_mpc} to be compatible with the differentiable optimal control framework presented in \cref{sec:doc}.
The nominal controller seeks to optimize the \textit{parameterized} safe nominal trajectory $\bar{\boldsymbol{\tau}}$ in the absence of
disturbances by solving the following optimization:
\begin{problem}[Differentiable Nominal MPC]\label{prob:diff_nominal_mpc}
\begin{gather*}
\begin{aligned}
\boldsymbol{\bar{\tau}}(\bar{\theta}) = \argmin_{\boldsymbol{\tau}} \bar{J}(\boldsymbol{\tau}, \bar{\theta}) \defeq {}& \sum_{k = 0}^{N - 1} \bar{\ell}(x_k, u_k,\bar{\theta}) + \bar{\phi}(x_N,\bar{\theta}),
\end{aligned} \\
\begin{aligned}
\text{subject to} \ \ &x_{k + 1} = f(x_k, u_k,\bar{\theta}), \quad x_0 = \bar{x}_t,  \\
&u_k \in \mathbb{U}, \quad x_k \in \bar{\mathbb{X}}(\bar{\theta}) \subset \mathbb{X}.
\end{aligned}
\end{gather*}
In contrast with the original definition of \cref{prob:nominal_mpc}, the dynamics $f$, running cost $\bar{\ell}$, and terminal cost $\bar{\phi}$ have been parameterized with $\bar{\theta}$ to denote the fact that these parameters are made tunable through the use of differentiable optimization.
Similarly, $\bar{\mathbb{X}}$ represents the tightened state constraint set, now parameterized by $\bar{\theta}$. This parameterization enables the online determination of the effective size of $\bar{\mathbb{X}}$ through differentiable optimization, instead of hand tuning it using, e.g., offline data.
\end{problem}

In our algorithm, the parameterization of $\bar{\mathbb{X}}$ with $\bar{\theta}$ is enabled through the adoption of barrier states for enforcing safety constraints and their penalizations in the cost function. In the use of DBaS as described in \cref{subsec: dbas}, an additional state within the cost function is introduced. For example, for the cost function $J$, we will assume that it has the partitioned form $\hat{J} = J + \sum_{k = 0}^{N} q_b b_k^2$, with $q_b > 0$ being a tunable parameter quantifying the strength of the barrier. This enables the online determination of the effective size of the feasible region $\bar{\mathbb{X}}$ through adaptation of $q_b$ based on task performance and predicted safety of the true system. Nonetheless, it should be noted that the proposed method is general and this is one specific choice that the authors find to work very well in practice. We will show different controller designs and solutions in \cref{sec:experiments}.

Consequently, the differentiable ancillary MPC is defined by bringing the original problem formulation from \cref{prob:ancillary_mpc} into the form of \cref{prob:parameterized_oc}:
\begin{problem}[Differentiable Ancillary MPC] \label{prob:diff_ancillary_mpc}
\begin{gather*}
\begin{aligned}
\boldsymbol{\tau}^*(\theta) = \argmin_{\boldsymbol{\tau}} J(\boldsymbol{\tau}, \theta, t) \defeq & \sum_{k = 0}^{N - 1} \ell(x_k, u_k, \bar{x}_k, \bar{u}_k, \theta) \\
 & \quad + \phi(x_N, \bar{x}_N, \theta),
 \end{aligned}\\
\text{subject to} \ \  x_{k + 1} = f(x_k, u_k, \theta), \quad x_0 = x_t, \quad u_k \in \mathbb{U} ,
\end{gather*}
where the dynamics $f$, running cost $\ell$, and terminal cost $\phi$ have been parameterized with $\theta$.
\end{problem}

It is worth highlighting again that the safety constraint does not appear in \cref{prob:diff_ancillary_mpc} as the state equation represents the safety-embedded system through the use of the DBaS as mentioned earlier. This has the effect of increasing the state dimension of the problem but does not affect the algorithm computationally as the input dimension is unchanged (DDP-based methods scale with the size of the control input but not the state input~\cite{mayne1966second}).
Note that this means that $\theta$ includes the DBaS tunable parameters such as $\gamma$ and $\alpha$ for the DBaS feedback and the relaxation of the barrier condition for the recursive feasibility of the algorithm when needed.

Through the introduction of the \ac{DBaS}, an additional degree of freedom that helps determine the tightness of the constraint satisfaction is added to the control problem. Therefore, our method is able to adjust the conservativeness of the nominal controller while also adapting the tube shape and size by updating the cost parameters of the ancillary controller based on the disturbances encountered. This allows for an efficient computational and engineering framework for robust control. Namely, the nominal controller can be tuned for task completion in the well-understood deterministic case, as is typical in standard nonlinear control design. The adaptive ancillary controller can then be used to augment the nominal controller for robustness, responding to disturbances when necessary. 

Next, we propose an algorithm that applies the DOC methodology presented in \cref{sec:doc} to the real-time tuning of tube-based controllers of the form given by \cref{prob:diff_nominal_mpc} and \cref{prob:diff_ancillary_mpc}. In order to optimize both the nominal and ancillary controller, we propose to use a loss function of the form
\begin{align}\label{eq:dt_mpc_loss}
L(\boldsymbol{\tau}^*(\theta), \bar{\boldsymbol{\tau}}(\bar{\theta})) = \norm{\boldsymbol{x}^*(\theta) - \bar{\boldsymbol{x}}(\bar{\theta})}_2^2 + \norm{\boldsymbol{b}^*(\theta)}_2^2,
\end{align}
but modify the loss in the experiments as necessary to capture task-specific objectives. Notably, this choice of objective is beneficial as it captures the primary goal of tube-based MPC, which is to drive the true state $x$ towards the nominal state $\bar{x}$ while maintaining safety of the true system. Furthermore, this choice allows the parameters of the nominal MPC to be updated based on the expected performance of the ancillary MPC, allowing the nominal MPC to respond to the environment in an optimal manner. The general differentiable tube-based MPC algorithm is presented in \cref{alg:dt_mpc} and is a straightforward application of \cref{alg:doc}.

\begin{algorithm}[h]
\caption{Differentiable Tube-based Model Predictive Control (DT-MPC)} \label{alg:dt_mpc}
\KwIn{Initial nominal parameters $\bar{\theta}$ and ancillary parameters $\theta$, learning rate~$\eta$, task horizon $H$}

$\bar{x}_0 \gets x_0$;

\For{$t = 0, \ldots, H$}{

    $\bar{\boldsymbol{\tau}}(\bar{\theta}) \gets$ Solve \cref{prob:diff_nominal_mpc} starting from $\bar{x}_t$;
    
    $\boldsymbol{\tau}^*(\theta) \gets$ Solve \cref{prob:diff_ancillary_mpc} starting from $x_t$;

    $\nabla_{\bar{\theta}} L, \nabla_{\theta} L \gets$ Compute gradients of $L(\boldsymbol{\tau}^*(\theta), \bar{\boldsymbol{\tau}}(\bar{\theta}))$ using \cref{alg:doc};

    \tcp{Gradient descent step}
    $\bar{\theta} \gets \bar{\theta} - \eta \nabla_{\bar{\theta}} L$; \quad $\theta \gets \theta - \eta \nabla_{\theta} L$; \quad 

    \tcp{True and nominal dynamics}
    $x_{t + 1} \gets f(x_t, u_t^*) + w_t$; \quad $\bar{x}_{t + 1} \gets f(\bar{x}_t, \bar{u}_t)$;
}
\end{algorithm}

\Cref{fig:architecture} visualizes the information and gradient flow of the proposed architecture. In summary, \cref{alg:dt_mpc} solves the nominal MPC (\cref{prob:diff_nominal_mpc}) from the current nominal state $\bar{x}_t$ with fixed nominal parameters $\bar{\theta}$, which consist of the DBaS parameters and task-dependent cost function weights (Line 3). The solution $\bar{\boldsymbol{\tau}}$ is passed to the ancillary MPC and \cref{prob:diff_ancillary_mpc} is solved starting from the current state $x_t$ for fixed ancillary parameters $\theta$, which consist of the DBaS parameters and the cost function weights that determine the tracking cost of the feedback controller (Line 4). Gradients of a task-dependent loss function are computed with respect to both the nominal and the ancillary parameters through \cref{alg:doc} (Line 5), and a gradient update is performed (Line 6). Finally, the ancillary control is sent to the true system, and the controller receives an observation of the next state perturbed by the disturbance $w_t$. Meanwhile, the nominal system propagates forward in the absence of noise or uncertainty (Line 7), and the process is repeated.

The proposed algorithm improves upon and solves problems within previous contributions in multiple directions. On the one hand, for complex problems, many parameters need to be tuned and carefully selected, e.g., through trial and error, in order for the controller to perform the desired task well. The proposed DT-MPC provides a theoretically sound and practical approach to auto-tune the different parameters involved to achieve a high level of autonomy. On the other hand, not only does the proposed algorithm provide robustness to the safe trajectory optimization problem through tube-based MPC, but it also provides adaptability and auto-tuning of the DBaS weight in the DBaS-iLQR MPC formulation by \citet{cho2023model}. Rather than keeping the weight fixed, which might result in an overly-conservative solution that withstands disturbances but prevents task completion, the proposed DT-MPC results in an adaptive, time-varying weight that is used to solve the safety-critical MPC and therefore better approximates the original optimal control problem. This is especially true when the nominal controller uses a poor model that might guide the ancillary controller to violate the safety condition. Moreover, through the use of relaxed barriers along with DDP, the work by \citet{feller2016relaxed} is generalized to the nonlinear case while auto-tuning the relaxation parameter $\alpha$ online. This allows us to recover the original barrier when needed.

\section{Experiments}
\label{sec:experiments}

The generality of the proposed DT-MPC is established through benchmarks on five nonlinear robotics systems subject to highly non-convex constraints such as dense obstacle fields. 
Furthermore, we present a hardware experiment showing the ability of DT-MPC to adapt to an out-of-distribution test case.
In the experiments that follow, the nominal MPC is tuned to perform the task successfully and then the algorithms are deployed on the \textit{true} system, without further tuning. This puts the proposed framework to the test, especially in comparison to the non-adaptive, nonlinear tube-based MPC. This suite of experiments showcases how the differentiable framework enables robust control through online adaptation of the necessary parameters to accomplish the task while maintaining safety.
\Cref{table:percentages} summarizes the results of the five simulation experiments, detailing the overall task completion percentage as well as the percentage of safety violations for each algorithm.
Specific numerical information related to the experiments, such as the parameterization for each controller as well as timing comparisons between algorithms, is provided in \cref{appendix:experiments}.

\begin{table*}[th]
\begin{center}
\begin{tabular}{ c|g|r|g|r|g|r|g|r|g|r| }
\hhline{~|*{10}{-}|}
  & \multicolumn{2}{c|}{\cellcolor{gray!5}Dubins Vehicle} & \multicolumn{2}{c|}{\cellcolor{gray!5}Quadrotor} &
  \multicolumn{2}{c|}{\cellcolor{gray!5}Robot Arm} & 
  \multicolumn{2}{c|}{\cellcolor{gray!5}Cheetah} &
  \multicolumn{2}{c|}{\cellcolor{gray!5}Quadruped} \\
\hhline{~|*{10}{-}|}
                &  Successes & Violations  & Successes & Violations & Successes & Violations & Successes & Violations & Successes & Violations\\
\hline
 NT-MPC &         14\% & 0\% & 14\% & 20\%   &  0\% &  56\%  & 26\% & 4\% & 20\% & 0\% \\
 \textbf{DT-MPC (ours)}  & \textbf{100\%} &  0\% & \textbf{76\%} &  \textbf{4\%}   &   \textbf{78\%} &  \textbf{10\%}  & \textbf{70\%} & \textbf{0\%} &\textbf{ 64\%} & 0\% \\
 \hline
\end{tabular}
\vspace{5pt}
\caption{Success and safety violation percentage for each algorithm over 50 trials per task. For successes, higher is better, while for violations, lower is better. Success is defined as arriving close to a target state, while a violation is defined as colliding with an obstacle or violating the safety constraints. The magnitude of disturbances are upper-bounded by $0.05$, $0.1$, $0.1$, $0.05$ and $0.05$ for each system, respectively.
}
\label{table:percentages}
\end{center}
\end{table*}

\subsection{Dubins Vehicle}

As an illustrative example, a Dubins vehicle task is set up where the goal is to reach a target state of $(10, 10)$~\SI{}{\meter} while avoiding obstacles, as shown in \cref{fig:dubins_comparison}. The vehicle starts at the origin facing the upper-right direction and, at every timestep, receives disturbances sampled uniformly from the range $[-0.05, 0.05]$ in both the xy position as well as its orientation. The nominal MPC parameters are kept fixed, while the ancillary MPC is allowed to adapt through minimization of the loss defined in \cref{eq:dt_mpc_loss}.

\begin{figure}[h]
\centering
\includegraphics[width=0.75\linewidth]{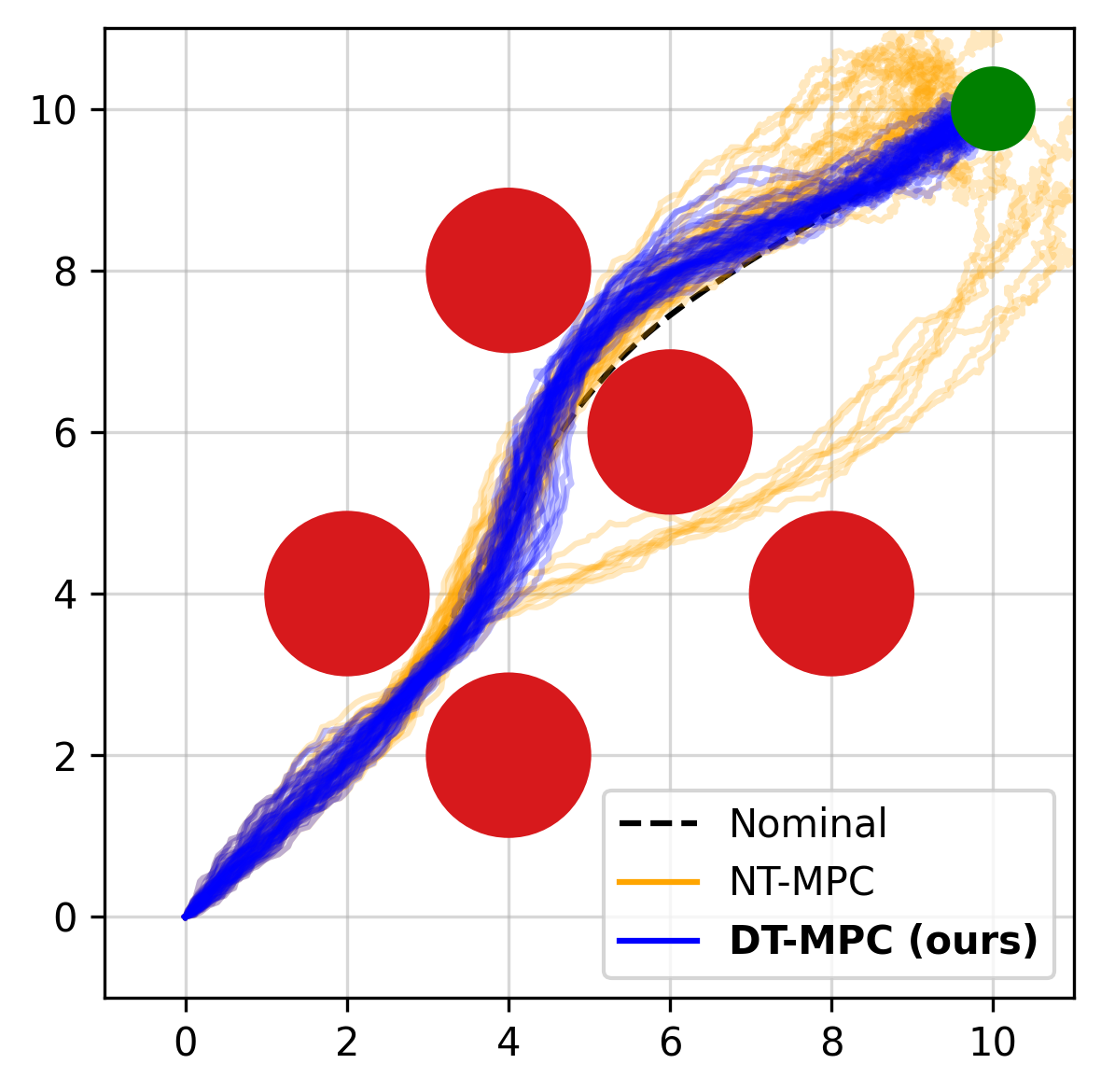}
\caption{Controlled Dubins vehicle trajectories subject to large noise. NT-MPC trajectories diverge from the nominal trajectory and the uncertainty increases over time. Meanwhile, DT-MPC adapts to the environment, maintaining safety and robust task performance.}
\label{fig:dubins_comparison}
\end{figure}

Controlled trajectories under both tube-based MPC algorithms for 50 different disturbance realizations are plotted in \cref{fig:dubins_comparison}. DT-MPC bounds the true system within a safer tube around the nominal trajectory, and the trajectories remain on the same side of the central obstacle. Additionally, the size of the tube is updated based on safety (proximity to obstacles), while the NT-MPC tube grows as the trajectory proceeds.

While both algorithms remain safe and avoid collisions (see \cref{table:percentages}), only DT-MPC is able to complete the task the majority of the time. This is attributed to the fact that the tube around the nominal trajectory is both tighter and safer allowing for more robust control. Furthermore, a qualitatively beneficial emergent behavior is observed where the tube of trajectories becomes tighter around the nominal trajectory when the state is most unsafe.

\subsection{Quadrotor}

The second experiment is a quadrotor navigating through a dense field of spherical obstacles, where the goal is to reach the target location of $(10, 10, 10)$~\SI{}{\meter}, starting from the origin, in the presence of large disturbances. The system experiences disturbances in the position and Euler angles sampled uniformly from the range $[-0.01, 0.01]$, while disturbances in the linear and angular velocities are sampled from a larger range of $[-0.1, 0.1]$ --- this choice emulates large unmodeled forces and moments in the dynamics. Like the Dubins vehicle experiment, the nominal MPC parameters are fixed, while the ancillary MPC is adapted through minimization of \cref{eq:dt_mpc_loss}, balancing tracking performance with safety.

From \cref{fig:quadrotor_comparison} and \cref{table:percentages}, it can be seen that NT-MPC is unable to maintain robustness, resulting in 20\% of the trajectories hitting obstacles or diverging from the nominal system, and only 14\% of trajectories are able to solve the task. On the other hand, the proposed DT-MPC bounds the true system within a safer tube around the nominal trajectory by tuning the ancillary MPC in real-time, drastically increasing the success rate to 76\% with only 4\% violations.

\subsection{Robot Arm}
\begin{figure}
\centering
\vspace{0pt}
\includegraphics[width=0.75\linewidth]{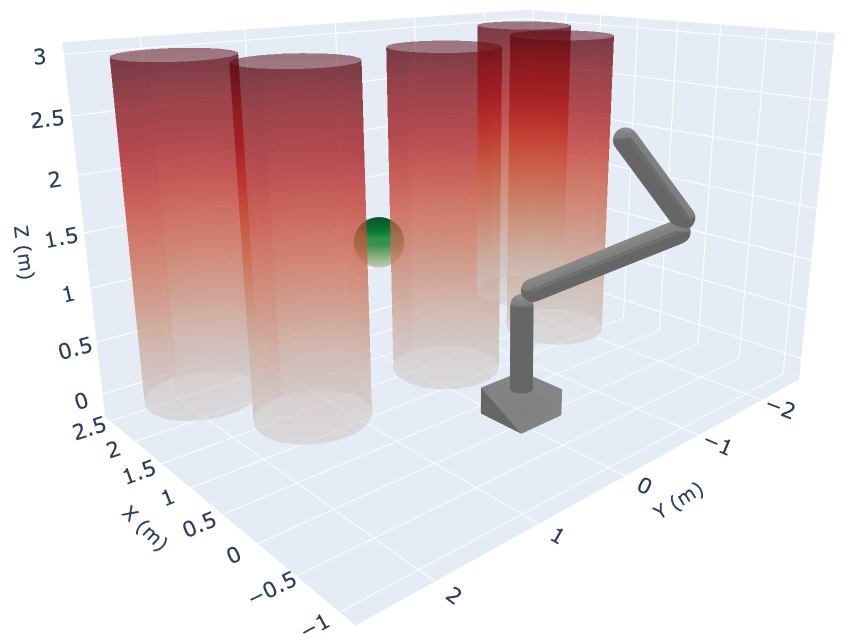}
\caption{Environment for the robot arm task.}
\label{fig:robot_arm_env}
\end{figure}

A 6-DOF torque-controlled robot arm with three links is used to test the ability of the proposed approach to adapt both the nominal and ancillary controller to respond to disturbances. A visualization for the task is given in \cref{fig:robot_arm_env}. The goal is to bring the end effector to the location of $(2, 0, 1)$~\SI{}{\meter}, starting from a random initial orientation, in the presence of large disturbances and obstacles.

Disturbances in the angles and angular velocities at every timestep are sampled uniformly from the ranges $[-0.01, 0.01]$~\SI{}{\radian} and $[-0.1, 0.1]$~\SI{}{\radian\per\second}, respectively, which, like the quadrotor experiment, corresponds to very large unmodeled forces and moments. For this task, the nominal controller (tuned for deterministic task completion) is too aggressive for the magnitude of disturbances received, leading to a large number of failures even when controlled using NT-MPC. Meanwhile, the proposed DT-MPC is able to \emph{adapt} the nominal controller to be less aggressive (by tuning the barrier state cost weights) while simultaneously tuning the ancillary controller to be more robust.

\subsection{Cheetah}

\begin{figure*}[h]
    \centering
    {\includegraphics[width=\textwidth]{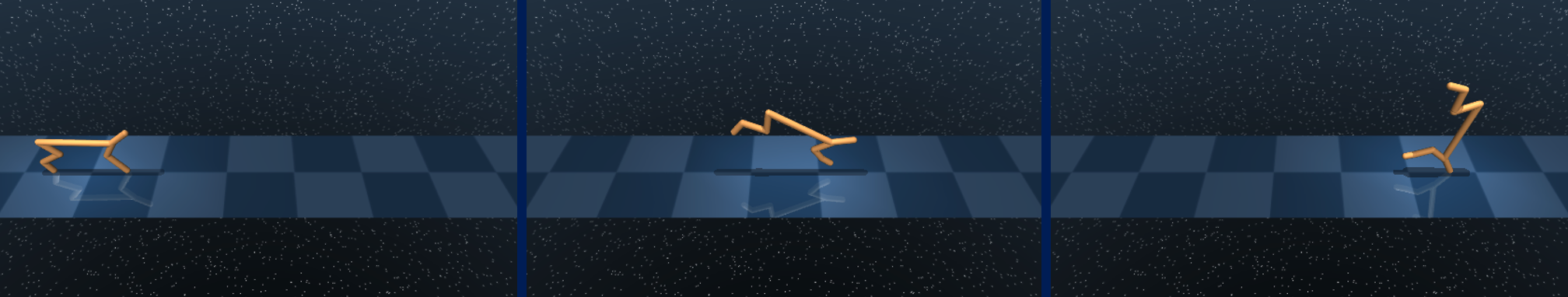}
    \subcaption{Nonlinear tube-based MPC. The lack of adaptation causes the cheetah to flip over.} 
    \label{fig:cheetah_no_tuning}}\par\vfill
    {\includegraphics[width=\textwidth]{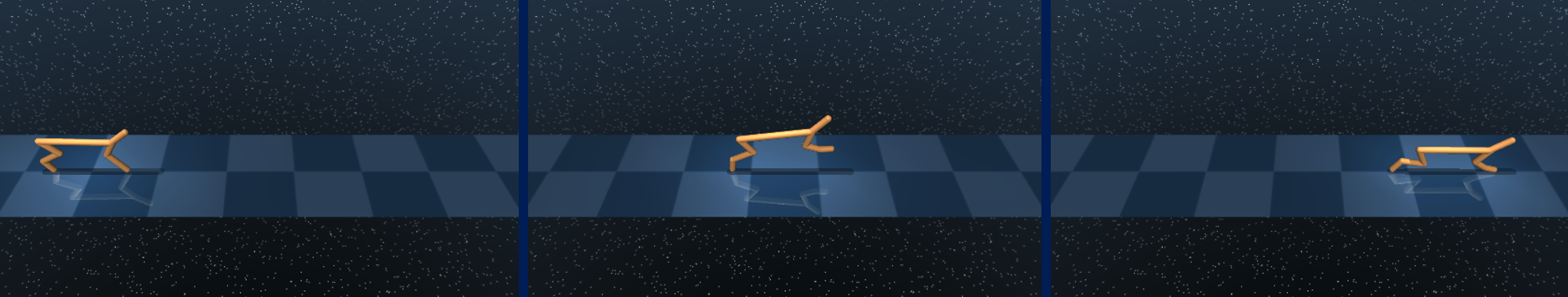}
    \subcaption{Differentiable tube-based MPC. Adaptation produces a stable running behavior and allows the cheetah to reach the target safely.}
    \label{fig:cheetah_with_tuning}}
    \caption{Comparison of tube-based MPC approaches on the DeepMind Control Suite cheetah robotics system \cite{tunyasuvunakool2020}.}
    \label{fig:cheetah}
\end{figure*}

In the next experiment, we design a locomotion task for the cheetah model provided by the DeepMind Control Suite~\citep{tunyasuvunakool2020}. The objective of the controller is to drive the cheetah to a target position of \SI{5}{\meter} in \SI{3}{\second}, requiring an average velocity of \SI{1.67}{\meter/\second}, while maintaining safe operation. Safety is enforced by constraining the pitch angle of the cheetah to stay within the range $[-\pi/4, \pi/4]$ --- when the pitch angle becomes larger than these bounds, it is difficult in general for the cheetah to recover, especially in the presence of large noise. Disturbances are injected into the velocity states and are sampled uniformly from $[-0.05, 0.05]$. Additional control noise is added sampled from $\mathcal{N}(0, 0.02)$. 

For DT-MPC, we use the loss $L = \norm{\boldsymbol{p}_x^* - \bar{\boldsymbol{p}}_x}_2^2 + \norm{\boldsymbol{b}^*}_2^2$, where $p_x$ is the state corresponding to the x-position, and allow the running cost and barrier state parameters of the nominal MPC and the ancillary MPC to be adapted online. Notably, the terminal cost is kept fixed for the nominal MPC to ensure the task objective information is not lost. The loss function is selected to only penalize the x-position as otherwise all of the states are weighted equally and the ancillary MPC prefers to default to a stable (yet safe) configuration rather than complete the task.

The results in \cref{table:percentages} show that, while NT-MPC fails to reach the target in the majority of the cases and occasionally violates the safety of the system, the proposed DT-MPC is able to adapt the entire robust MPC architecture for task completion while maintaining safe control.
A sample visualization of one trial per algorithm is provided in \cref{fig:cheetah}.

\subsection{Quadruped}

\begin{figure*}[h]
    \centering
    {\includegraphics[width=\textwidth]{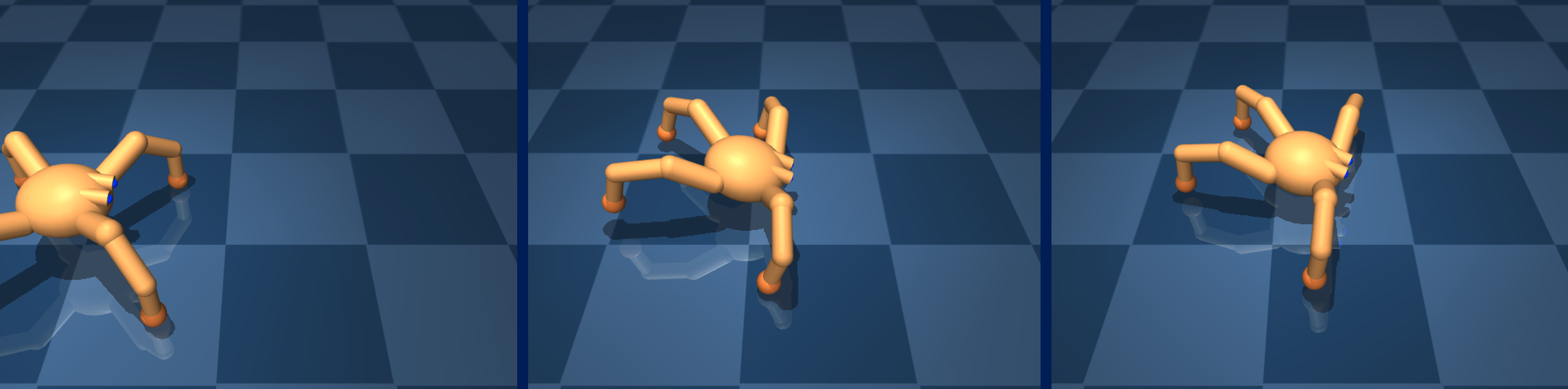}
    \subcaption{Nonlinear tube-based MPC. The quadruped gets stuck because the true state diverges from the nominal due to the disturbances.}
    \label{fig:quadruped_nt_mpc}}\par\vfill
    {\includegraphics[width=\textwidth]{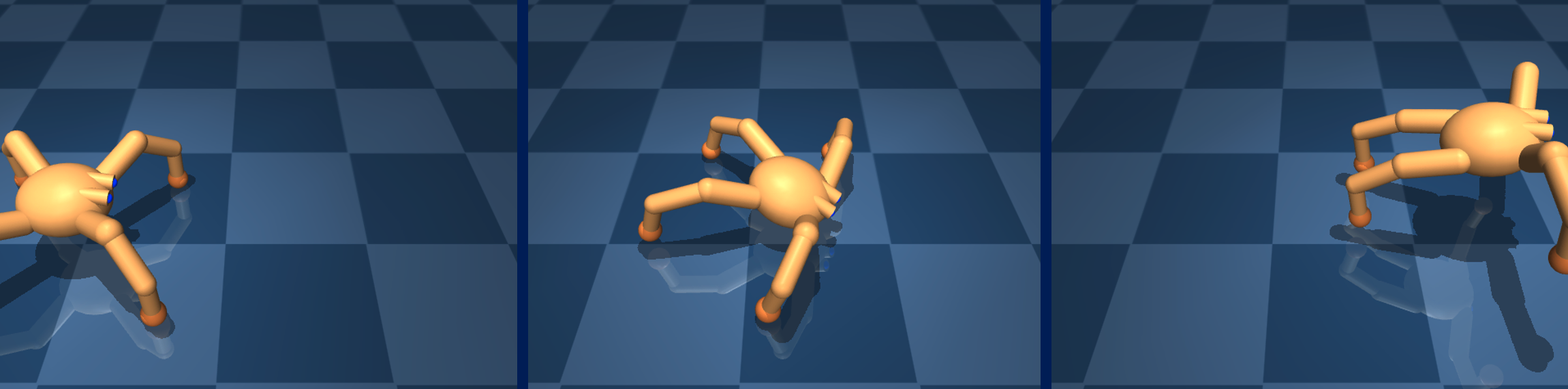}
    \subcaption{Differentiable tube-based MPC. The adaptation enables the quadruped to successfully reach the target despite large disturbances.}
    \label{fig:quadruped_dt_mpc}}
    \caption{Comparison of tube-based MPC approaches on the DeepMind Control Suite quadruped system \cite{tunyasuvunakool2020}.}
    \label{fig:quadruped}
\end{figure*}

The final simulation experiment is a locomotion task using the quadruped model provided by the DeepMind Control Suite~\citep{tunyasuvunakool2020}. The system has 56 states and 12 controls, and the objective is to drive the quadruped towards a target position of \SI{2.5}{\meter} in \SI{2}{\second} requiring an average velocity of \SI{1.25}{\meter/\second}. Dynamical uncertainty is simulated by perturbing the velocity states with a disturbance sampled from the range $[-0.05, 0.05]$ and introducing control noise sampled from $\mathcal{N}(0, 0.01)$. Safety is determined by constraining the Euler angles of the system such that the quadruped does not flip over ($[-\pi/2, \pi/2]$ for the roll and yaw angles and $[-1.0, 1.0]$ for the pitch angle). \Cref{fig:quadruped} provides a visualization of the model and task.

The loss for DT-MPC is selected to be the same as the cheetah experiment, namely $L = \norm{\boldsymbol{p}_x^* - \bar{\boldsymbol{p}}_x}_2^2 + \norm{\boldsymbol{b}^*}_2^2$. Similarly, both the nominal and ancillary MPC running cost and barrier state parameters are adapted online.

Like the robot arm and cheetah experiment, the nominal controller in NT-MPC has no awareness of the true task performance. Therefore, the system controlled under NT-MPC has difficulty reaching the target. While the deterministic nominal trajectory reaches the target state during every trial, the ancillary controller cannot keep up with the desired aggressive jumping maneuver due to the disturbances. This causes the gap between the nominal trajectory and true state to grow over time. Meanwhile, the system controlled under DT-MPC allows the nominal controller to adapt based on the tracking performance of the ancillary MPC. Furthermore, the ancillary MPC is adapted online to keep the system safe. This results in successful aggressive jumping maneuvers such as the one shown in the bottom right of \cref{fig:quadruped}. Overall, DT-MPC remains safe while increasing the task success rate by over 200\% (20\% success rate for NT-MPC vs. 64\% success rate for DT-MPC).

\subsection{Hardware Experiment - Robotarium}

\begin{figure*}[h]
    \centering
    {\hspace*{\fill}\begin{subfigure}[b]{0.33\textwidth}
        \includegraphics[width=\textwidth]{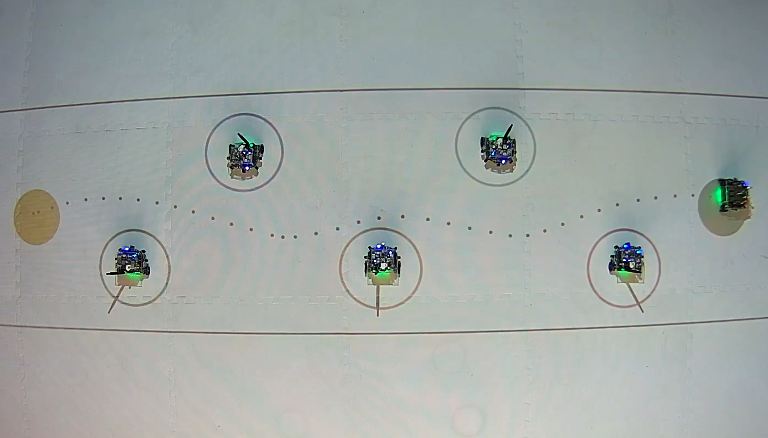}
        \caption{NT-MPC --- stationary obstacles}
        \label{fig:robotarium_nt_mpc_stationary}
    \end{subfigure}\hfill
    \begin{subfigure}[b]{0.33\textwidth}
        \includegraphics[width=\textwidth]{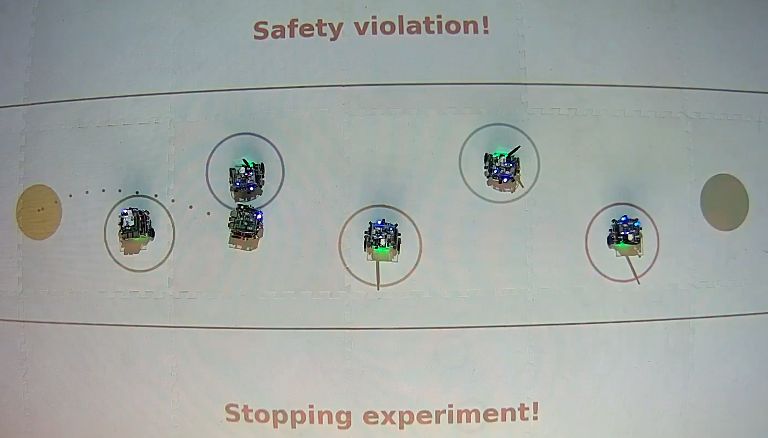}
        \caption{NT-MPC --- moving obstacles}
        \label{fig:robotarium_nt_mpc_moving}
    \end{subfigure}\hspace*{\fill}}\par\vspace{0.5em}
    {\includegraphics[width=\textwidth]{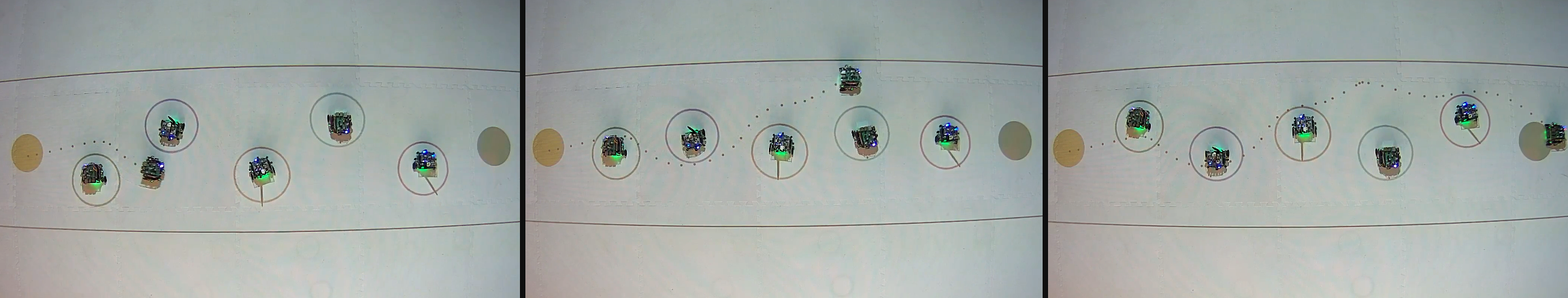}
    \subcaption{DT-MPC --- moving obstacles}
    \label{fig:robotarium_dt_mpc}}
    \caption{Comparison of tube-based MPC approaches on the Robotarium hardware platform \cite{wilson2020robotarium}. \textbf{(a)} NT-MPC successfully reaches the target when the task is in-distribution. \textbf{(b)} NT-MPC is unable to adapt when the task is out-of-distribution and quickly violates the safety of the system. \textbf{(c)} DT-MPC enables fast adaptation to a novel task, resulting in safe, robust control.}
    \label{fig:robotarium}
\end{figure*}

Finally, we implement the proposed methodology on the Robotarium (\cref{fig:robotarium}), a state-of-the-art, remotely accessible robotics hardware platform for multi-agent control \cite{wilson2020robotarium}.
DT-MPC is applied for the safe, real-time control of a differential drive robot under a shifting task distribution.
The controller is tuned offline on a given task distribution but needs to \emph{adapt} online to a new, never-before-seen task.
In the parlance of machine learning, the controller needs to generalize to a novel task that is outside of the training distribution.

The goal of the controller is to safely navigate through a corridor while avoiding the other agents, modeled as circular obstacles.
As a baseline, the NT-MPC controller is tuned on a fixed task distribution where the other agents remain stationary throughout the experimental trial.
NT-MPC is robust to disturbances due to both modeling error and process noise and can reach the target state successfully (\cref{fig:robotarium_nt_mpc_stationary}).

During test time, the other agents are allowed to move with some small random velocity.
The non-adaptive NT-MPC is unable to respond to the change in task distribution effectively, resulting in safety violations early in the trial (\cref{fig:robotarium_nt_mpc_moving}).
We attempt 10 retries of the NT-MPC controller, and during all of the trials, the agent violates the safety of the system by driving too close to the first two obstacles.
Meanwhile, the DT-MPC agent quickly finds a suitable controller tuning that maintains safety while ensuring robust task completion (\cref{fig:robotarium_dt_mpc}).
The addition of online differentiable optimization for adaptation of the cost function parameters allows the DT-MPC agent to adapt to the new environment quickly and efficiently. The JAX-based Python implementation of our method runs at over \SI{50}{\hertz} on the Robotarium --- we expect further speedups can be achieved through a lower-level implementation (e.g., in C++).

\section{Conclusion}
\label{sec:conclusion}

This work has proposed a general algorithm for differentiable optimal control with a unified perspective derived from applying the implicit function theorem to the lower-level control problem.
We have detailed how to apply the methodology to the real-time tuning of tube-based MPC controllers, illustrating a principled manner for which to learn the problem parameters of any nonlinear control algorithm.
Our framework allows for the robust, adaptive, and real-time control of nonlinear systems in the presence of non-convex constraints and large uncertainties.
The generality of the method is verified for the robust control of multiple nonlinear systems both in simulation and on real hardware.
Furthermore, we provide theoretical guarantees on the numerical precision of our approach and show how our work generalizes state-of-the-art differentiable control algorithms.
We emphasize the fact that our proposed architecture enables the real-time tuning of \emph{any} tube-based MPC algorithm.
Promising future extensions of the work include learning the dynamical uncertainty and adapting the model used by the MPC controllers.

\section*{Acknowledgments}
This work is supported by NASA under the ULI grant 80NSSC22M0070 and by NSF CPS award 1932288.

\footnotesize
\bibliographystyle{plainnat}
\bibliography{references}

\normalsize
\begin{appendices}
\crefalias{section}{appendix}

\section{Proof of Proposition 2}
\label{appendix:optimality_conditions}

\setcounter{theorem}{1}
\begin{proposition}[Optimality conditions of \cref{prob:parameterized_oc}]
Define the real-valued function $\mathcal{L}$ for \cref{prob:parameterized_oc}, called the \emph{Lagrangian}, as
\begin{align}
\begin{aligned}
    \mathcal{L}(\boldsymbol{z}, \theta) = {}& \sum_{t = 0}^{N - 1} \ell(x_k, u_k, \theta) + \lambda_{k + 1}^\top (f(x_k, u_k, \theta) - x_{k + 1}) \\
    & + \lambda_0^\top (\xi(\theta) - x_0) + \phi(x_N, \theta), \label{eq:lagrangian}
\end{aligned}
\end{align}
where $\lambda_k \in \mathbb{R}^{n_x}$, $k = 0, 1, \ldots, N$ are the Lagrange multipliers for the dynamics and initial condition constraints, and $\boldsymbol{z} \defeq (\boldsymbol{\tau}, \boldsymbol{\lambda}) = (\lambda_0, x_0, u_0, \ldots, \lambda_N, x_N)$ for notational compactness.

Let $\boldsymbol{\tau}^*$ be a solution to \cref{prob:parameterized_oc} for fixed parameters~$\theta$. Then, there exists Lagrange multipliers~$\boldsymbol{\lambda}^*$ which together with $\boldsymbol{\tau}^*$ satisfy $\nabla_{\boldsymbol{z}} \mathcal{L}(\boldsymbol{z}^*, \theta) = 0$.
\end{proposition}
\begin{proof}
Taking the gradient of the Lagrangian \cref{eq:lagrangian} with respect to each argument yields the following set of equations:
\begin{align*}
    \nabla_{\lambda_0} \mathcal{L} &= \xi(\theta) - x_0^* , \\
    \nabla_{\lambda_{k + 1}} \mathcal{L} &= f(x_k^*, u_k^*, \theta) - x_{k + 1}^* , \\
    \nabla_{x_k} \mathcal{L} &= \ell_{x_k} + f_{x_k}^\top \lambda_{k + 1} - \lambda_k , \\
    \nabla_{u_k} \mathcal{L} &= \ell_{u_k} + f_{u_k}^\top \lambda_{k + 1} , \\
    \nabla_{x_N} \mathcal{L} &= \phi_{x} - \lambda_N ,
\end{align*}
where functions are evaluated at the points $x_k^*, u_k^*, \theta$ along the solution $\boldsymbol{\tau}^*$ and the abbreviated notation $\ell_{x_k}$, $f_{u_k}$, etc. is used for gradients and partial derivatives. The first two equations are equal to zero because a feasible solution satisfies the initial condition and dynamics constraints of \cref{prob:parameterized_oc}. The remaining equations, when set equal to zero, yield a set of backward equations defining the optimal Lagrange multipliers~$\boldsymbol{\lambda}^*$:
\begin{align*}
    \lambda_k^* &= \ell_{x_k} + f_{x_k}^\top \lambda_{k + 1}^* , \quad \lambda_N^* = \phi_{x}, \\
    0 &= \ell_{u_k} + f_{u_k}^\top \lambda_{k + 1}^*. 
\end{align*}
These facts together imply the statement of the proposition. These conditions $\nabla_{\boldsymbol{z}} \mathcal{L}(\boldsymbol{z}^*, \theta) = 0$ are known as the \ac{KKT} conditions, and are equivalent to the discrete-time Pontryagin maximum/minimum principle (PMP) (see, e.g., Appendix C of \citet{jin2020pontryagin}).
\end{proof}

\section{Proof of Proposition 3}
\label{appendix:implicit_derivative}

\begin{proposition}[Implicit derivative of \cref{prob:parameterized_oc}]
Let $\boldsymbol{\tau}^*$ be a solution to \cref{prob:parameterized_oc} for fixed parameters~$\theta$. By \cref{proposition:optimality_conds}, \cref{thm:ift} holds with $F = \nabla_{\boldsymbol{z}} \mathcal{L}$. Furthermore, the Jacobian $\pdv{\boldsymbol{z}^*}{\theta}$ is given as
\begin{align*}
\pdv{}{\theta}\boldsymbol{z}^*(\theta) = - \mathcal{L}_{\boldsymbol{z}\boldsymbol{z}}^{-1} \mathcal{L}_{\boldsymbol{z}\theta} .
\end{align*}
\end{proposition}
\begin{proof}
Let \cref{proposition:optimality_conds} hold for $\boldsymbol{z}^* = (\boldsymbol{\tau}^*, \boldsymbol{\lambda}^*)$. Furthermore, assume the Hessian $\mathcal{L}_{\boldsymbol{z}\boldsymbol{z}}$ evaluated at $\boldsymbol{z}^*$ is nonsingular, which can be ensured through, e.g., Levenberg–Marquardt regularization~\citep{todorov2005generalized}. Since $\nabla_{\boldsymbol{z}} \mathcal{L}(\boldsymbol{z}^*, \theta) = 0$, the gradient $\nabla_{\boldsymbol{z}} \mathcal{L}$ satisfies the conditions of \cref{thm:ift}, and the statement of the proposition follows naturally.
\end{proof}

\section{Proof of Corollary 4}
\label{appendix:pdp}

\begin{corollary}[Pontryagin differentiable programming~\citep{jin2020pontryagin}]
Let the conditions of \cref{proposition:optimality_conds} and \cref{proposition:implicit_derivative} hold. Then, the differentiable Pontryagin conditions ((Eq. (13) of \citet{jin2020pontryagin}) are equivalent to solving the linear system \cref{eq:implicit_derivative}.
\end{corollary}
\begin{proof}
First, let us state the differentiable Pontryagin conditions~\citep{jin2020pontryagin} in the notation of this paper:
\begin{equation}\label{eq:pdp}
\begin{aligned}
    \pdv{x_0}{\theta} &= \xi_\theta , \\
    \pdv{x_{k + 1}}{\theta} &= f_{x_k} \pdv{x_k}{\theta} + f_{u_k} \pdv{u_k}{\theta} + f_{\theta_k}, \\
    \pdv{\lambda_k}{\theta} &= \mathcal{L}_{xx}^{(k)} \pdv{x_k}{\theta} + \mathcal{L}_{xu}^{(k)} \pdv{u_k}{\theta} + f_{x_k}^\top \pdv{\lambda_{k + 1}}{\theta} + \mathcal{L}_{x\theta}^{(k)}, \\
    0 &= \mathcal{L}_{ux}^{(k)} \pdv{x_k}{\theta} + \mathcal{L}_{uu}^{(k)} \pdv{u_k}{\theta} + f_{u_k}^\top \pdv{\lambda_{k + 1}}{\theta} + \mathcal{L}_{u\theta}^{(k)}, \\
    \pdv{\lambda_N}{\theta} &= \phi_{xx} \pdv{x_N}{\theta} + \phi_{x\theta} ,
\end{aligned}
\end{equation}
where all derivatives are evaluated along the solution~$\boldsymbol{z}^*$. The notation $\mathcal{L}_{ux}^{(k)}$ denotes the block of the Hessian of the Lagrangian \cref{eq:lagrangian} corresponding to $u_k^*, x_k^*$, see below. Let us write explicitly the form of the Jacobian $\pdv{\boldsymbol{z}^*}{\theta} = (\ldots, \pdv{\lambda_k}{\theta}, \pdv{x_k}{\theta}, \pdv{u_k}{\theta}, \pdv{\lambda_{k + 1}}{\theta}, \ldots, \pdv{\lambda_N}{\theta}, \pdv{x_N}{\theta})$, the Hessian of the Lagrangian
\setlength{\arraycolsep}{2pt}
\begin{equation}\label{eq:L_zz}
    \mathcal{L}_{\boldsymbol{z}\boldsymbol{z}} = \begin{bNiceMatrix}[margin,first-row,last-col]
         \textcolor{red}{\lambda_k} & \textcolor{red}{x_k} & \textcolor{red}{u_k} & \textcolor{red}{\lambda_{k + 1}} & & \textcolor{red}{x_N} \\
         \ddots & -I & & & & & \textcolor{red}{\lambda_k} \\
         -I & \mathcal{L}_{xx}^{(k)} & \mathcal{L}_{xu}^{(k)} & f_{x_k}^\top & & & \textcolor{red}{x_k} \\
         & \mathcal{L}_{ux}^{(k)} & \mathcal{L}_{uu}^{(k)} & f_{u_k}^\top & & & \textcolor{red}{u_k} \\
         & f_{x_k} & f_{u_k} & & & & \textcolor{red}{\lambda_{k + 1}} \\
         & & & & \ddots & -I & \\
         & & & & -I & \phi_{xx} & \textcolor{red}{x_N} 
    \end{bNiceMatrix},
\end{equation}
and the Hessian $\mathcal{L}_{\boldsymbol{z}\theta} = (\xi_\theta, \ldots \mathcal{L}_{x\theta}^{(k)}, \mathcal{L}_{u\theta}^{(k)}, f_{\theta_k}, \ldots, \phi_{x\theta})$. Expanding the matrix multiplication $\mathcal{L}_{\boldsymbol{z}\boldsymbol{z}} \pdv{\boldsymbol{z}^*}{\theta} = - \mathcal{L}_{\boldsymbol{z}\theta}$ yields
\begin{equation*}
\begin{split}
-\pdv{x_0}{\theta} &= -\xi_\theta ,\\
-\pdv{\lambda_k}{\theta} + \mathcal{L}_{xx}^{(k)} \pdv{x_k}{\theta} + \mathcal{L}_{xu}^{(k)} \pdv{u_k}{\theta} + f_{x_k}^\top \pdv{\lambda_{k + 1}}{\theta} &= -\mathcal{L}_{x\theta}^{(k)}, \\
\mathcal{L}_{ux}^{(k)} \pdv{x_k}{\theta} + \mathcal{L}_{uu}^{(k)} \pdv{u_k}{\theta} + f_{u_k}^\top \pdv{\lambda_{k + 1}}{\theta} &= -\mathcal{L}_{u\theta}^{(k)}, \\
f_{x_k} \pdv{x_k}{\theta} + f_{u_k} \pdv{u_k}{\theta} - \pdv{x_{k + 1}}{\theta} &= -f_{\theta_k}, \\
-\pdv{\lambda_N}{\theta} + \phi_{xx} \pdv{x_N}{\theta} &= -\phi_{x\theta} ,
\end{split}
\end{equation*}
which are equivalent to \cref{eq:pdp} above.
\end{proof}

\section{Proof of Theorem 5}
\label{appendix:doc}
\begin{theorem}[Differentiable Optimal Control]
Let $\boldsymbol{z}$ denote the augmented vector consisting of $\boldsymbol{\tau}$ and $\boldsymbol{\lambda}$. 
In addition, let the conditions of \cref{proposition:optimality_conds} and \cref{proposition:implicit_derivative} hold. Then, the gradient of the loss $L$ with respect to $\theta$ is given by
\begin{align*}
    \nabla_\theta L(\boldsymbol{z}^*(\theta)) = \mathcal{L}_{\theta \boldsymbol{z}} \delta \boldsymbol{z} ,
\end{align*}
where the vector $\delta \boldsymbol{z}$ is given by solving the linear system
\begin{align*}
    \mathcal{L}_{\boldsymbol{z}\boldsymbol{z}} \delta \boldsymbol{z} = -\nabla_{\boldsymbol{z}} L .
\end{align*}
\end{theorem}
\begin{proof}
Start by stating the expression for the gradient $\nabla_\theta L$ and applying \cref{proposition:implicit_derivative}:
\begin{align*}
    \nabla_\theta L(\boldsymbol{z}^*(\theta)) &= \pdv{\boldsymbol{z}^*(\theta)}{\theta}^\top \nabla_{\boldsymbol{z}} L \\
    &= -\mathcal{L}_{\theta\boldsymbol{z}} \mathcal{L}_{\boldsymbol{z}\boldsymbol{z}}^{-1} \nabla_{\boldsymbol{z}} L.
\end{align*}
The last line shows this gradient can be computed either by solving $-\mathcal{L}_{\boldsymbol{z}\boldsymbol{z}}^{-1} \mathcal{L}_{\boldsymbol{z}\theta}$ (\cref{corollary:pdp}) or $-\mathcal{L}_{\boldsymbol{z}\boldsymbol{z}}^{-1} \nabla_{\boldsymbol{z}} L$. We adopt here the second choice --- call the solution $\delta \boldsymbol{z} = -\mathcal{L}_{\boldsymbol{z}\boldsymbol{z}}^{-1} \nabla_{\boldsymbol{z}} L$, equivalently $\mathcal{L}_{\boldsymbol{z}\boldsymbol{z}} \delta\boldsymbol{z} = -\nabla_{\boldsymbol{z}} L$. Substituting $\delta\boldsymbol{z}$ back into the expression for the gradient completes the proof. 
\end{proof}

\section{Derivation of Algorithm 1}
\label{appendix:doc_algorithm}

The derivation of \cref{alg:doc} is very similar to the derivation of DDP~\citep{mayne1966second,jacobson1970differential} --- in fact, it is the same, the only difference being the gradient that is considered. The resultant algorithm consists of a backward pass and forward pass procedure summarized by the equations in \cref{alg:doc_bpass} and \cref{alg:doc_fpass}, with derivation below. Diacritic tildes (e.g., $\widetilde{Q}_x^{(k)}$) are used to emphasize quantities that are unique from the usual DDP equations.

\begin{algorithm}
\caption{DOC Backward Pass}\label{alg:doc_bpass}
\KwIn{Derivatives of $\mathcal{L}$ (equivalently $f$, $\ell$, $\phi$, and $\xi$) and $L$ along the solution $\boldsymbol{z}^*$}
\KwOut{Derivatives and gains $\widetilde{\boldsymbol{V}}_x, \boldsymbol{V}_{xx}, \widetilde{\boldsymbol{k}}, \boldsymbol{K}$}

$\widetilde{V}_x^{(N)} \gets \nabla_{x_N} L$; \quad $V_{xx}^{(N)} \gets \phi_{xx}$;

\For{$k = N - 1, \ldots, 0$}{
    \tcp{$Q$ derivatives} 
    $\widetilde{Q}_x^{(k)} \gets \nabla_{x_k} L + f_{x_k}^\top \widetilde{V}_x^{(k + 1)}$; \\
    $\widetilde{Q}_u^{(k)} \gets \nabla_{u_k} L + f_{u_k}^\top \widetilde{V}_x^{(k + 1)}$; \\
    $Q_{xx}^{(k)} \gets \mathcal{L}_{xx}^{(k)} + f_{x_k}^\top V_{xx}^{(k + 1)} f_{x_k}$; \\
    $Q_{ux}^{(k)} \gets \mathcal{L}_{ux}^{(k)} + f_{u_k}^\top V_{xx}^{(k + 1)} f_{x_k} = (Q_{xu}^{(k)})^\top$; \\
    $Q_{uu}^{(k)} \gets \mathcal{L}_{uu}^{(k)} + f_{u_k}^\top V_{xx}^{(k + 1)} f_{u_k}$; \\
    \tcp{Control gains} 
    $\widetilde{k}^{(k)} \gets -\big( Q_{uu}^{(k)} \big)^{-1} \widetilde{Q}_u^{(k)}$; \\
    $K^{(k)} \gets -\big( Q_{uu}^{(k)} \big)^{-1} Q_{ux}^{(k)}$; \\
    \tcp{$V$ derivatives} 
    $\widetilde{V}_{x}^{(k)} \gets \widetilde{Q}_x^{(k)} + Q_{xu}^{(k)} \widetilde{k}^{(k)}$; \\
    $V_{xx}^{(k)} \gets Q_{xx}^{(k)} + Q_{xu}^{(k)} K^{(k)}$;
}
\end{algorithm}

\begin{algorithm}
\caption{DOC Forward Pass}\label{alg:doc_fpass}
\KwIn{Outputs from \cref{alg:doc_bpass} and parameter derivatives $\xi_\theta$, $\boldsymbol{f}_{\theta}$, $\boldsymbol{\mathcal{L}}_{\theta x}$, $\boldsymbol{\mathcal{L}}_{\theta u}$, $\phi_{\theta x}$}
\KwOut{Gradient of upper-level loss $\nabla_\theta L$}

$\delta{x_0} \gets 0$;

$\delta{\lambda_0} \gets \widetilde{V}_x^{(0)}$;

$\nabla_\theta L \gets \xi_\theta^\top \delta{\lambda_0}$;

\For{$k = 0, \ldots, N - 1$}{

    $\delta{u_k} \gets \widetilde{k}^{(k)} + K^{(k)} \delta{x_k}$;
    
    $\delta{x_{k + 1}} \gets f_{x_k} \delta{x_k} + f_{u_k} \delta{u_k}$;
    
    $\delta{\lambda_{k + 1}} \gets \widetilde{V}_x^{(k + 1)} + V_{xx}^{(k + 1)} \delta{x_{k + 1}}$;\\
    \tcp{Accumulate gradient}
    $\nabla_\theta L \gets \nabla_\theta L + \mathcal{L}_{\theta x}^{(k)} \delta{x_k} + \mathcal{L}_{\theta u}^{(k)} \delta{u_k} + f_{\theta_k}^\top \delta{\lambda_{k + 1}}$; 
}
$\nabla_\theta L \gets \nabla_\theta L + \phi_{\theta x} \delta{x_N}$;
\end{algorithm}

The goal will be to derive an efficient algorithm for solving
\begin{align*}
    \mathcal{L}_{\boldsymbol{z}\boldsymbol{z}} \delta \boldsymbol{z} = -\nabla_{\boldsymbol{z}} L ,
\end{align*}
as given in \cref{thm:doc}.

Using the block form of $\mathcal{L}_{\boldsymbol{z}\boldsymbol{z}}$ given in \cref{eq:L_zz} and noting that $\nabla_{\boldsymbol{z}} L = (0, \ldots, \nabla_{x_k} L, \nabla_{u_k} L, \ldots, \nabla_{x_N} L)$, expanding the matrix multiplication above yields the following system of equations:
\begin{subequations}
\begin{align}
    \delta x_0 &= 0 ,\\
    \delta{\lambda_k} &= \nabla_{x_k} L + \mathcal{L}_{xx}^{(k)} \delta{x_k} + \mathcal{L}_{xu}^{(k)} \delta{u_k} + f_{x_k}^\top \delta{\lambda_{k + 1}}, \label{eq:dx_t} \\
    0 &= \nabla_{u_k} L + \mathcal{L}_{ux}^{(k)} \delta{x_k} + \mathcal{L}_{uu}^{(k)} \delta{u_k} + f_{u_k}^\top \delta{\lambda_{k + 1}}, \label{eq:du_t} \\
    \delta{x_{k + 1}} &= f_{x_k} \delta{x_k} + f_{u_k} \delta{u_k}, \label{eq:lin_dyn} \\
    \delta{\lambda_N} &= \phi_{xx} \delta{x_N} + \nabla_{x_N} L .
\end{align}
\end{subequations}

Through an inductive argument, we will show $\delta{\lambda_k}$ has the form
\begin{equation}
    \delta{\lambda_k} = \widetilde{V}_x^{(k)} + V_{xx}^{(k)} \delta{x_k} .\label{eq:dlambda_t}
\end{equation}

\cref{eq:dlambda_t} holds at time $N$ by taking $\widetilde{V}_x^{(N)} = \nabla_{x_N} L$ and $V_{xx}^{(N)} = \phi_{xx}$. Next, assume \cref{eq:dlambda_t} holds at time $k + 1$, namely
\begin{align*}
    \delta{\lambda_{k + 1}} &= \widetilde{V}_x^{(k + 1)} + V_{xx}^{(k + 1)} \delta{x_{k + 1}} \\
    &= \widetilde{V}_x^{(k + 1)} + V_{xx}^{(k + 1)} f_{x_k} \delta{x_k} + V_{xx}^{(k + 1)} f_{u_k} \delta{u_k},
\end{align*}
where we have substituted in the linear dynamics \cref{eq:lin_dyn}. Substituting into \cref{eq:du_t} yields
\begin{equation}\label{eq:du_t_opt}
\delta{u_k} = \tilde{k}^{(k)} + K^{(k)} \delta{x_k}  ,
\end{equation}
where
\begin{align*}
    \tilde{k}^{(k)} &\defeq -\big(Q_{uu}^{(k)}\big)^{-1} \widetilde{Q}_{u}^{(k)},\\
    K^{(k)} &\defeq - \big(Q_{uu}^{(k)}\big)^{-1} Q_{ux}^{(k)}, \\
    \widetilde{Q}_{u}^{(k)} &\defeq \nabla_{u_k}L + f_{u_k}^\top \widetilde{V}_{x}^{(k + 1)}, \\
    Q_{uu}^{(k)} &\defeq \mathcal{L}_{uu}^{(k)} + f_{u_k}^\top V_{xx}^{(k + 1)} f_{u_k}, \\
    Q_{ux}^{(k)} &\defeq \mathcal{L}_{ux}^{(k)} + f_{u_k}^\top V_{xx}^{(k + 1)} f_{x_k} .
\end{align*}

Finally, substituting $\delta{u_k}$ and $\delta{\lambda_{k + 1}}$ into \cref{eq:dx_t} yields
\begin{equation*}
  \delta{\lambda_k} = \widetilde{V}_x^{(k)} + V_{xx}^{(k)} \delta{x_k},
\end{equation*}
with
\begin{align*}
  \widetilde{V}_x^{(k)} &\defeq \widetilde{Q}_{x}^{(k)} + Q_{xu}^{(k)} \widetilde{k}^{(k)}, \\
  V_{xx}^{(k)} &\defeq Q_{xx}^{(k)} + Q_{xu}^{(k)} K^{(k)}, \\
  \widetilde{Q}_{x}^{(k)} & \defeq \nabla_{x_k}L + f_{x_k}^\top \widetilde{V}_{x}^{(k + 1)}, \\
  Q_{xx}^{(k)} &\defeq \mathcal{L}_{xx}^{(k)} + f_{x_k}^\top V_{xx}^{(k + 1)} f_{x_k}, \\
  Q_{xu}^{(k)} &\defeq \mathcal{L}_{xu}^{(k)} + f_{x_k}^\top V_{xx}^{(k + 1)} f_{u_k} = (Q_{ux}^{(k)})^\top ,
\end{align*}
which proves \cref{eq:dlambda_t} holds at time $k$. By induction, it holds for all $k = N, N - 1, \ldots, 0$. This motivates the backward-forward nature of \cref{alg:doc}, as the $Q$ and $V$ derivatives can be computed backwards-in-time, then the solution $\delta \boldsymbol{z}$ is given by forward application of \cref{eq:du_t_opt}, \cref{eq:lin_dyn}, and \cref{eq:dlambda_t}.

Finally, the computation of the gradient $\nabla_\theta L$ is given as
\begin{equation*}
    \nabla_\theta L = \mathcal{L}_{\theta \boldsymbol{z}} \delta \boldsymbol{z} .
\end{equation*}
Multiplication of $\mathcal{L}_{\theta \boldsymbol{z}}$ with $$\delta \boldsymbol{z} = (\delta{\lambda_0}, \ldots, \delta{x_k}, \delta{u_k}, \delta{\lambda_{k + 1}}, \ldots, \delta{x_N})$$ yields the sum
\begin{align*}
    \nabla_\theta L ={}& \xi_\theta^\top \delta{\lambda_0} + \ldots \\
    &+ \mathcal{L}_{\theta x}^{(k)} \delta{x_k} + \mathcal{L}_{\theta u}^{(k)} \delta{u_k} + f_{\theta_k}^\top \delta{\lambda_{k + 1}} + \ldots \\
    &+ \phi_{\theta x} \delta{x_N} ,
\end{align*}
which shows how the gradient can be accumulated during the forward pass of \cref{alg:doc}.

\section{Proof of Corollary 6}
\label{appendix:diff_mpc}
\begin{corollary}[Diff-MPC~\citep{amos2018differentiable}]
Diff-MPC is equivalent to using a Gauss-Newton approximation when solving \cref{eq:d_z}.
\end{corollary}
\begin{proof}
Diff-MPC was originally derived by \citet{amos2018differentiable} starting with iLQR and then implicitly differentiating with respect to the iLQR parameters $A_k$, $B_k$, $Q_k$, $R_k$, etc., using matrix calculus. We present an alternative derivation that generalizes~\citep{amos2018differentiable} to consider arbitrary dynamics and cost parameters through applying \cref{alg:doc} using a Gauss-Newton approximation of the Hessian --- this approximation is advantageous computationally as it only requires first-order dynamics derivatives. Furthermore, adopting the Gauss-Newton approximation is equivalent to the iLQR algorithm~\citep{giftthaler2018family}.

Adopting the notation of this work, the Diff-MPC equations are given by solving the system
$$K \delta \boldsymbol{z} = -\nabla_{\boldsymbol{z}} L,$$
with the KKT matrix $K$ given by
\setlength{\arraycolsep}{2pt}
\begin{equation*}\label{eq:K}
    K = \begin{bNiceMatrix}[margin,first-row,last-col]
         \textcolor{red}{\lambda_k} & \textcolor{red}{x_k} & \textcolor{red}{u_k} & \textcolor{red}{\lambda_{k + 1}} & & \textcolor{red}{x_N} \\
         \ddots & -I & & & & & \textcolor{red}{\lambda_k} \\
         -I & \ell_{xx}^{(k)} & \ell_{xu}^{(k)} & f_{x_k}^\top & & & \textcolor{red}{x_k} \\
         & \ell_{ux}^{(k)} & \ell_{uu}^{(k)} & f_{u_k}^\top & & & \textcolor{red}{u_k} \\
         & f_{x_k} & f_{u_k} & & & & \textcolor{red}{\lambda_{k + 1}} \\
         & & & & \ddots & -I & \\
         & & & & -I & \phi_{xx} & \textcolor{red}{x_N} 
    \end{bNiceMatrix},
\end{equation*}
This matrix should look familiar to the form of $\mathcal{L}_{\boldsymbol{z}\boldsymbol{z}}$ given in \cref{eq:L_zz}. Consider the second-order derivatives of the Lagrangian that are necessary when applying \cref{alg:doc} to invert $\mathcal{L}_{\boldsymbol{z}\boldsymbol{z}}$:
\begin{align*}
\mathcal{L}_{xx}^{(k)} &= \ell_{xx}^{(k)} + \lambda_{k + 1} \otimes f_{xx}^{(k)} , \\
\mathcal{L}_{ux}^{(k)} &= \ell_{ux}^{(k)} + \lambda_{k + 1} \otimes f_{ux}^{(k)} , \\
\mathcal{L}_{uu}^{(k)} &= \ell_{uu}^{(k)} + \lambda_{k + 1} \otimes f_{uu}^{(k)} , \\
\end{align*}
with the $\otimes$ operator representing tensor contraction. Dropping the final term of each equation (and thus avoiding their expensive computation) corresponds to a Gauss-Newton approximation of the Hessian, and we have that $\mathcal{L}_{\boldsymbol{z}\boldsymbol{z}} = K$. Note that, importantly, this approximation does not change the structure of the algorithm itself. This shows that Diff-MPC can be derived through our framework by applying a Gauss-Newton Hessian approximation in \cref{alg:doc}.
\end{proof}

\section{Theorem 7 --- Jacobian estimate error}
\label{appendix:jac_error}

We start by defining some necessary quantities for computing the Jacobian $\frac{\partial \boldsymbol{z}}{\partial \theta}$ using \cref{proposition:implicit_derivative}. To avoid excessive fraction notation, we abbreviate $\partial \boldsymbol{z} \defeq \frac{\partial \boldsymbol{z}}{\partial \theta}$ and similarly $\partial x_k \defeq \frac{\partial x_k}{\partial \theta}$ and $\partial u_k \defeq \frac{\partial u_k}{\partial \theta}$. As shown in \cref{corollary:pdp}, computing $\partial \boldsymbol{z}$ amounts to solving the system of equations \cref{eq:pdp}. This system has the solution:
\begin{align*}
\partial u_k &= - (Q_{uu}^{(k)})^{-1} Q_{u\theta}^{(k)} - (Q_{uu}^{(k)})^{-1} Q_{ux}^{(k)} \partial x_k , \\
\partial x_{k + 1} &= f_{x_k} \partial x_k + f_{u_k} \partial u_k + f_{\theta_k} ,
\end{align*}
with initial condition $\partial x_0 = \xi_{\theta}$. The matrix-valued functions $Q_{uu}^{(k)}$ and $Q_{ux}^{(k)}$ are defined the same as in the derivation of \cref{alg:doc} presented in \cref{appendix:doc_algorithm}. $Q_{u\theta}^{(k)}$ is defined by
\begin{align*}
    Q_{u\theta}^{(k)} &\defeq \mathcal{L}_{u\theta}^{(k)} + f_{u_k}^\top V_{x\theta}^{(k + 1)} + f_{u_k}^\top V_{xx}^{(k + 1)} f_{\theta_k}, \\
    Q_{x\theta}^{(k)} &\defeq \mathcal{L}_{x\theta}^{(k)} + f_{x_k}^\top V_{x\theta}^{(k + 1)} + f_{x_k}^\top V_{xx}^{(k + 1)} f_{\theta_k}, \\
    V_{x\theta}^{(k)} &\defeq Q_{x\theta}^{(k)} - Q_{xu}^{(k)} (Q_{uu}^{(k)})^{-1} Q_{u\theta}^{(k)} .
\end{align*}
where $V_{x\theta}^{(k)}$ is solved backwards-in-time (much like the value function gradient $V_x^{(k)}$ and Hessian $V_{xx}^{(k)}$) starting with $V_{x\theta}^{(N)} = \phi_{x\theta}$.

In the following results, we drop the dependence of $Q_{uu}$, $Q_{ux}$, etc. on time $k$ for notational compactness. The norm of a matrix $\normt{A}$ is the operator norm unless otherwise specified. Let $A: \mathbb{R}^n \to \mathbb{R}^{m \times p}$ be any matrix-valued function. $\beta$-boundedness of $A$ implies $\norm{A(x)} \leq \beta$. Likewise, $L$-Lipschitzness of $A$ implies $\norm{A(y) - A(x)} \leq L \norm{y - x}$.

Next, let us state formally the necessary assumptions of \cref{thm:jac_error}. We use Greek letters for boundedness constants and capital Roman letters for Lipschitz constants. All of the following mathematical objects are functions defined in a neighborhood of the optimal trajectory $\boldsymbol{z}^*$. 
\setcounter{theorem}{7}\begin{assumption}
\label{assumptions:jac_error}
The following conditions are assumed to hold for all $\boldsymbol{z}$ with $\norm{\boldsymbol{z} - \boldsymbol{z}^*} \leq \epsilon$:
\begin{enumerate}
    \item $Q_{uu}$ is $K$-Lipschitz in $(x_t, u_t)$ and well-conditioned:
    $\norm{Q_{uu} v}~\geq~\alpha \norm{v}$ for all $v \in \mathbb{R}^{n_u}$. This implies $Q_{uu}^{-1}$ exists and $\norm{Q_{uu}^{-1}} \leq \frac{1}{\alpha}$.
    \item $Q_{ux}$ is $L$-Lipschitz and $\beta$-bounded.
    \item $Q_{u\theta}$ is $M$-Lipschitz and $\gamma$-bounded.
    \item $f_x$ is $A$-Lipschitz and $\zeta$-bounded.
    \item $f_u$ is $B$-Lipschitz and $\eta$-bounded.
    \item $f_\theta$ is $N$-Lipschitz and $\sigma$-bounded.
    \item $\xi_\theta$ is $\rho$-bounded.
\end{enumerate}
\end{assumption}

Now we will present the main theorem, which shows that the Jacobian estimate errors are bounded as a function of the trajectory error to the optimal solution.
\begin{manualtheorem}{7}[Jacobian estimate error]
\label{thm:jac_error_full}
Let the conditions of \cref{assumptions:jac_error} hold. Then, the error in using \cref{proposition:implicit_derivative} to compute the implicit derivatives of the control problem is upper bounded by
\begin{align*}
    \norm{\partial \hat{u}_k - \partial u_k^*} &\leq \sum_{t = 0}^{k} C_{k,t} \norm{\hat{\tau}_t - \tau_t^*} ,\\ 
    \norm{\partial \hat{x}_{k + 1} - \partial x_{k + 1}^*} &\leq \sum_{t = 0}^{k} D_{k+1,t} \norm{\hat{\tau}_t - \tau_t^*} ,
\end{align*}
where $\tau_t = (x_t, u_t)$ is the state-control pair at time $t$. The constants $C_{k, t}, D_{k, t} > 0$ are specific to the time step $k$ that the Jacobian errors are evaluated at, and are time-varying for $t = 0, 1, \ldots, k$.
\end{manualtheorem}

Before proving the full theorem, we begin by proving the following result which shows the Jacobians $\partial u_k$ and $\partial x_k$ remain bounded for all time.
\begin{lemma}[Boundedness of Jacobians]
\label{lemma:bounded_jac}
Let the conditions of \cref{assumptions:jac_error} hold. For all time steps $k = 0, \ldots, N - 1$, the norm of the Jacobians $\partial u_k$ and $\partial x_k$ are bounded:
\begin{align*}
    \norm{\partial u_k} \leq \mu_k, \quad \norm{\partial x_{k + 1}} \leq \nu_{k + 1},
\end{align*}
for $\mu_k = (\gamma + \beta \nu_k)/\alpha$ and $\nu_{k + 1} = \zeta \nu_k + \eta \mu_k + \sigma$ with $\nu_0 = \rho$.
\end{lemma}
\begin{proof}
Starting with the base case $k = 0$, we have $\partial x_0 = \xi_\theta$ so $\normt{\partial x_0} = \normt{\xi_\theta} \leq \rho \eqdef \nu_0$. Next, for $k \geq 0$ assume that $\normt{\partial x_k} \leq \nu_k$. We have that $\normt{\partial u_k} \leq \normt{Q_{uu}^{-1}} \normt{Q_{u\theta}} + \normt{Q_{uu}^{-1}} \normt{Q_{ux}} \normt{\partial x_k} \leq (\gamma + \beta \nu_k)/\alpha \eqdef \mu_k$. Finally, $\normt{\partial x_{k + 1}} \leq \normt{f_x}\normt{\partial x_k} + \normt{f_u}\normt{\partial u_k} + \normt{f_\theta} \leq \zeta \nu_k + \eta \mu_k + \sigma \eqdef \nu_{k + 1}$ and the result holds by recursion.
\end{proof}

\begin{proof}[Proof of \cref{thm:jac_error_full}]
To start, let us consider the base case with $k = 0$. Note that $\partial \hat{x}_{0} - \partial x_{0}^* = \xi_{\theta} - \xi_{\theta} = 0$ since the initial condition function $\xi(\theta)$ has no dependence on $x_0$. Therefore, we have $\normt{\partial \hat{x}_{0} - \partial x_{0}^*} = 0 \leq D_{0, 0} \normt{\hat{\tau}_0 - \tau_0^*}$ for any $D_{0, 0} \geq 0$, so take $D_{0, 0} = 0$ trivially.

Next, we analyze the control Jacobian error at time $k \geq 0$. It is assumed that the state Jacobian error for time $k$ can be bounded in the form $\normt{\partial \hat{x}_k - \partial x_k^*} \leq \sum_{t = 0}^{k - 1} D_{k, t} \normt{\hat{\tau}_t - \tau_t^*}$ according to the inductive hypothesis.

For notational convenience ``hat'' quantities such as $\hat{Q}_{uu}$ are assumed to be evaluated at $(\hat{x}_k, \hat{u}_k)$ while ``non-hat'' quantities such as $Q_{ux}$ are evaluated instead at the optimal state-control pair $(x_k^*, u_k^*)$. The general form of $\partial \hat{u}_k - \partial u_k^*$ can therefore be written as
\begin{gather*}
\begin{aligned}[t]
   \partial \hat{u}_k - \partial u_k^* ={}& -\hat{Q}_{uu}^{-1} \hat{Q}_{u\theta} - \hat{Q}_{uu}^{-1} \hat{Q}_{ux} \partial \hat{x}_k \\
    & \quad + Q_{uu}^{-1} Q_{u\theta} + Q_{uu}^{-1} Q_{ux} \partial x_k^*
\end{aligned} \\
\begin{aligned}[t]
    ={}& \hat{Q}_{uu}^{-1} \left( (Q_{u\theta} - \hat{Q}_{u\theta}) + \hat{Q}_{ux} (\partial x_k^* - \partial \hat{x}_k) + (Q_{ux} - \hat{Q}_{ux}) \partial x_k^* \right) \\
    & \quad + (Q_{uu}^{-1} - \hat{Q}_{uu}^{-1}) (Q_{u\theta} + Q_{ux} \partial x_k^*) ,
\end{aligned}
\end{gather*}
where we have used repeatedly the fact that $AB - CD = A(B - D) + (A - C)D$. The norm of this approximation can be upper-bounded by
\begin{gather*}
\begin{aligned}[t]
   \normt{\partial \hat{u}_k - \partial u_k^*} \leq{}& \normt{\hat{Q}_{uu}^{-1}} \big( \normt{Q_{u\theta} - \hat{Q}_{u\theta}} + \normt{\hat{Q}_{ux}} \normt{\partial x_k^* - \partial \hat{x}_k} \\
   & \quad\quad\quad + \normt{Q_{ux} - \hat{Q}_{ux}} \normt{\partial x_k^*} \big) \\
    & \quad + \normt{Q_{uu}^{-1} - \hat{Q}_{uu}^{-1}} \big(\normt{Q_{u\theta}} + \normt{Q_{ux}} \normt{\partial x_k^*}\big)
\end{aligned}\\
\begin{aligned}[t]
    ={}& \normt{\hat{Q}_{uu}^{-1}} \big( \normt{Q_{u\theta} - \hat{Q}_{u\theta}} + \normt{\hat{Q}_{ux}} \normt{\partial x_k^* - \partial \hat{x}_k} \\
    & \quad\quad\quad + \normt{Q_{ux} - \hat{Q}_{ux}} \normt{\partial x_k^*} \big) \\
    & \quad + \normt{Q_{uu}^{-1}} \normt{\hat{Q}_{uu} - Q_{uu}} \normt{\hat{Q}_{uu}^{-1}} \big(\normt{Q_{u\theta}} + \normt{Q_{ux}} \normt{\partial x_k^*}\big) ,
\end{aligned} 
\end{gather*}
where we have used the fact that $\normt{A^{-1} - B^{-1}} = \normt{A^{-1}}\normt{B - A}\normt{B^{-1}}$. Using \cref{assumptions:jac_error} and \cref{lemma:bounded_jac}, this yields the following upper bound:
\begin{gather*}
\begin{aligned}[t]
   \normt{\partial \hat{u}_k - \partial u_k^*} \leq{}& \underbrace{(\frac{M}{\alpha} + \frac{\nu_k L}{\alpha} + \frac{K(\gamma + \beta \nu_k)}{\alpha^2})}_{\eqdef C_{k, k}} \normt{\hat{\tau}_k - \tau_k^*} \\
   & \quad + \frac{\beta}{\alpha} \normt{\partial \hat{x}_k - \partial x_k^*}
\end{aligned}\\
\leq C_{k,k} \normt{\hat{\tau}_k - \tau_k^*} + \sum_{t = 0}^{k - 1} \underbrace{\frac{\beta}{\alpha} D_{k, t}}_{\defeq C_{k, t}} \normt{\hat{\tau}_t - \tau_t^*} = \sum_{t = 0}^{k} C_{k,t} \norm{\hat{\tau}_t - \tau_t^*} .
\end{gather*}
This shows the control Jacobian error is on the same order as the error with the optimal trajectory.

Next, let us analyze the state Jacobian error at time $k \geq 0$. The general form of $\partial \hat{x}_{k + 1} - \partial x_{k + 1}^*$ can be expresed as
\begin{align*}
\partial \hat{x}_{k + 1} - \partial x_{k + 1}^* ={}& \hat{f}_x \partial \hat{x}_k + \hat{f}_u \partial \hat{u}_k + \hat{f}_\theta \\
& \quad - f_x \partial x_k^* - f_u \partial u_k^* - f_\theta,
\end{align*}
which by similar arguments is bounded in norm by
\begin{align*}
\normt{\partial \hat{x}_{k + 1} - \partial x_{k + 1}^*} \leq{}& \normt{\hat{f}_x} \normt{\partial \hat{x}_k - \partial x_k^*} + \normt{\hat{f}_x - f_x} \normt{\partial x_k^*} \\
& + \normt{\hat{f}_u} \normt{\partial \hat{u}_k - \partial u_k^*} + \normt{\hat{f}_u - f_u} \normt{\partial u_k^*} \\
& + \normt{\hat{f}_\theta - f_\theta}.
\end{align*}
Under \cref{assumptions:jac_error} and \cref{lemma:bounded_jac}, we therefore have
\begin{gather*}
\begin{aligned}[t]
\normt{\partial \hat{x}_{k + 1} - \partial x_{k + 1}^*} \leq{}& \zeta \normt{\partial \hat{x}_k - \partial x_k^*} + \nu_k A \normt{\hat{\tau}_k - \tau_k^*} \\
& + \eta \normt{\partial \hat{u}_k - \partial u_k^*} + \mu_k B \normt{\hat{\tau}_k - \tau_k^*} \\
& + N \normt{\hat{\tau}_k - \tau_k^*} 
\end{aligned}\\
\begin{aligned}[t]
    \leq{}& \zeta \sum_{t = 0}^{k - 1} D_{k, t} \normt{\hat{\tau}_t - \tau_t^*} + \eta \sum_{t = 0}^{k} C_{k, t} \normt{\hat{\tau}_t - \tau_t^*} \\
    & + (\nu_k A + \mu_k B + N) \normt{\hat{\tau}_k - \tau_k^*} 
\end{aligned}\\
\begin{aligned}[t]
    \leq{}& \sum_{t = 0}^{k - 1} \underbrace{(\zeta D_{k, t} + \eta C_{k, t})}_{\defeq D_{k + 1, t}} \normt{\hat{\tau}_t - \tau_t^*} \\
    & + \underbrace{(\nu_k A + \mu_k B + N + \eta C_{k, k})}_{D_{k + 1, k}} \normt{\hat{\tau}_k - \tau_k^*} 
\end{aligned}\\
\implies \normt{\partial \hat{x}_{k + 1} - \partial x_{k + 1}^*} \leq \sum_{t = 0}^{k} D_{k + 1, t} \normt{\hat{\tau}_t - \tau_t^*} .
\end{gather*}
This completes the proof of \cref{thm:jac_error_full}.
\end{proof}

\section{Incorporating Control Constraints}
\label{appendix:control_constraints}

\cref{alg:doc} can be straightforwardly extended to incorporate control constraints. The most common case is that of box control limits $\underline{u} \leq u_k \leq \overline{u}$, which can be expressed as the general linear inequality constraint
\begin{equation*}
G u_k - \mu \leq 0 ,
\end{equation*}
with $G = \begin{bmatrix} -I \\ I \end{bmatrix}$ and $\mu = \begin{bmatrix} \underline{u} \\ -\overline{u} \end{bmatrix}$.

Since \cref{alg:doc} is applied at a solution $\boldsymbol{z}^*$, the control trajectory itself is fixed. Therefore, the control constraints can be partitioned into an active set and an inactive set (active meaning the equality holds). Let the active inequality constraints be given by $\widetilde{G} u_k - \widetilde{\mu} = 0$. This is equivalent to the condition $\delta u_k^{(i)} = 0$ if $u_k^{(i)} \in \{ \underline{u}, \overline{u} \}$~\citep{amos2018differentiable}, and is easily incorporated into the forward pass (\cref{alg:doc_fpass}).

General control inequality constraints of the form $g(u_k) \leq 0$ for $g$ differentiable can be handled similarly. This imposes the constraint $\pdv{\widetilde{g}}{u} \delta u_k = 0$ with $\widetilde{g}$ the active constraints, which can be seen by adding the control constraint to the Lagrangian and rederiving \cref{alg:doc}. In words, this condition restricts $\delta u_k$ to the null space of the Jacobian $\pdv{\widetilde{g}}{u}$, and can be incorporated into \cref{alg:doc} by reparameterizing the solution in terms of the QR decomposition of $\pdv{\widetilde{g}}{u}$ along the trajectory~\citep{grandia2023doc}.

\section{Gradient Error and Timing Comparisons}
\label{appendix:comparisons}
This supplementary section presents numerical comparisons between the proposed approach \cref{alg:doc} and the state-of-the-art works of \citet{amos2018differentiable}, \citet{dinev2022differentiable}, and \citet{jin2020pontryagin}. All algorithms are implemented in JAX~\citep{jax2018github} and run on a Mac M1 processor.

As described in \cref{subsec:precision}, the numerical precision of our approach is verified empirically on the Dubins vehicle, quadrotor, and robot arm systems. The Jacobian estimate error $\norm{J(\hat{\boldsymbol{z}}, \theta) - \frac{\partial \boldsymbol{z}^*}{\partial \theta}}$ is plotted as a function of the iterate error $\norm{\hat{\boldsymbol{z}} - \boldsymbol{z}^*}$ by running DDP and DOC on nominal trajectory optimization examples. We compare with the implicit derivative returned by Diff-MPC from \citep{amos2018differentiable} (equivalent to our Gauss-Newton approximation)  and an unrolled Jacobian that is computed by directly unrolling iterations of DDP through autodifferentiation. These results are presented in \cref{fig:dubins_error,fig:quad_error,fig:robot_arm_error} for the three systems, respectively.

Additionally, the average time in milliseconds over 100 iterations of each algorithm is presented below in \cref{fig:dubins_timing,fig:quad_timing,fig:robot_arm_timing}, followed by a breakdown showing which components of the algorithm contribute the most time (\cref{fig:dubins_timing_breakdown,fig:quad_timing_breakdown,fig:robot_arm_timing_breakdown}).

\newpage
\null
\vfill
\begin{figure}[H]
\centering
\begin{subfigure}[b]{0.5\textwidth}
    \centering
    \includegraphics[width=0.8\textwidth]{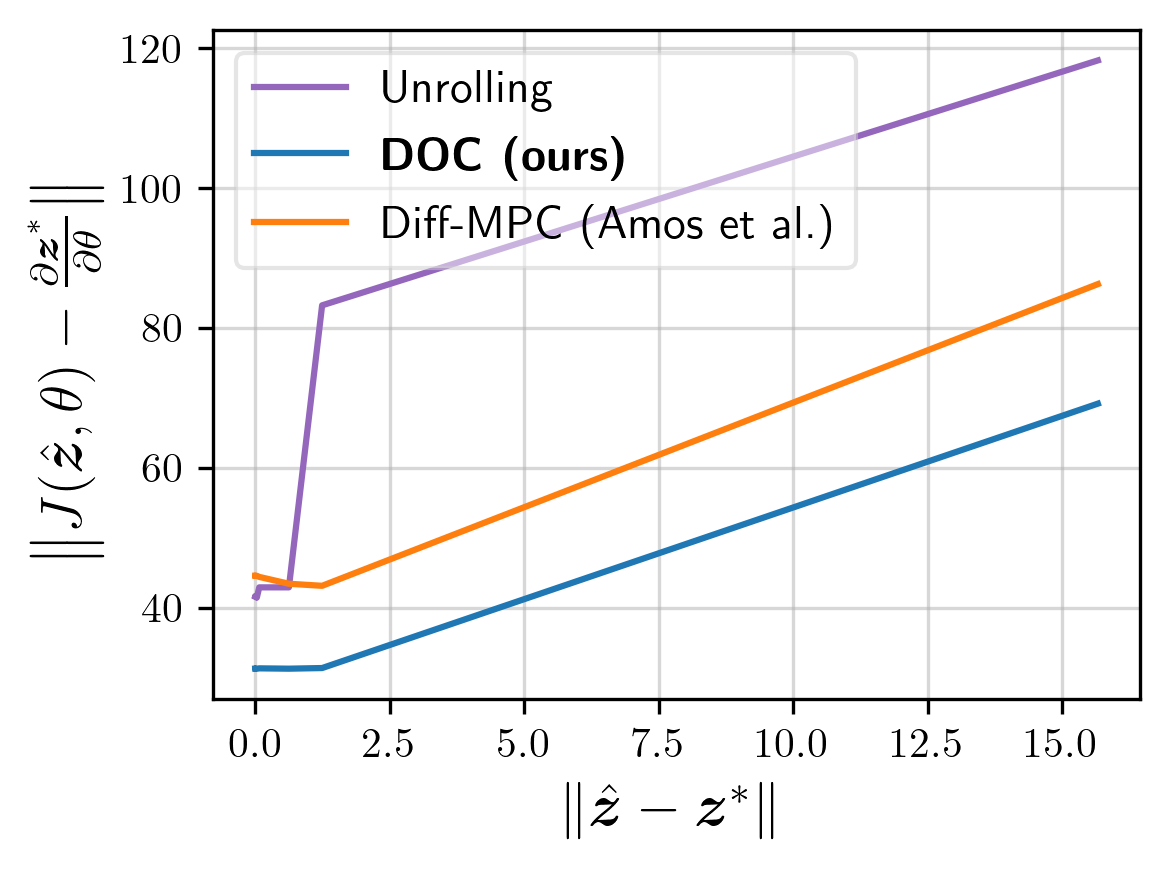}
    \caption{Jacobian estimate error.}
    \label{fig:dubins_error}
\end{subfigure}
\begin{subfigure}[b]{0.5\textwidth}
    \centering
    \includegraphics[width=1.0\textwidth]{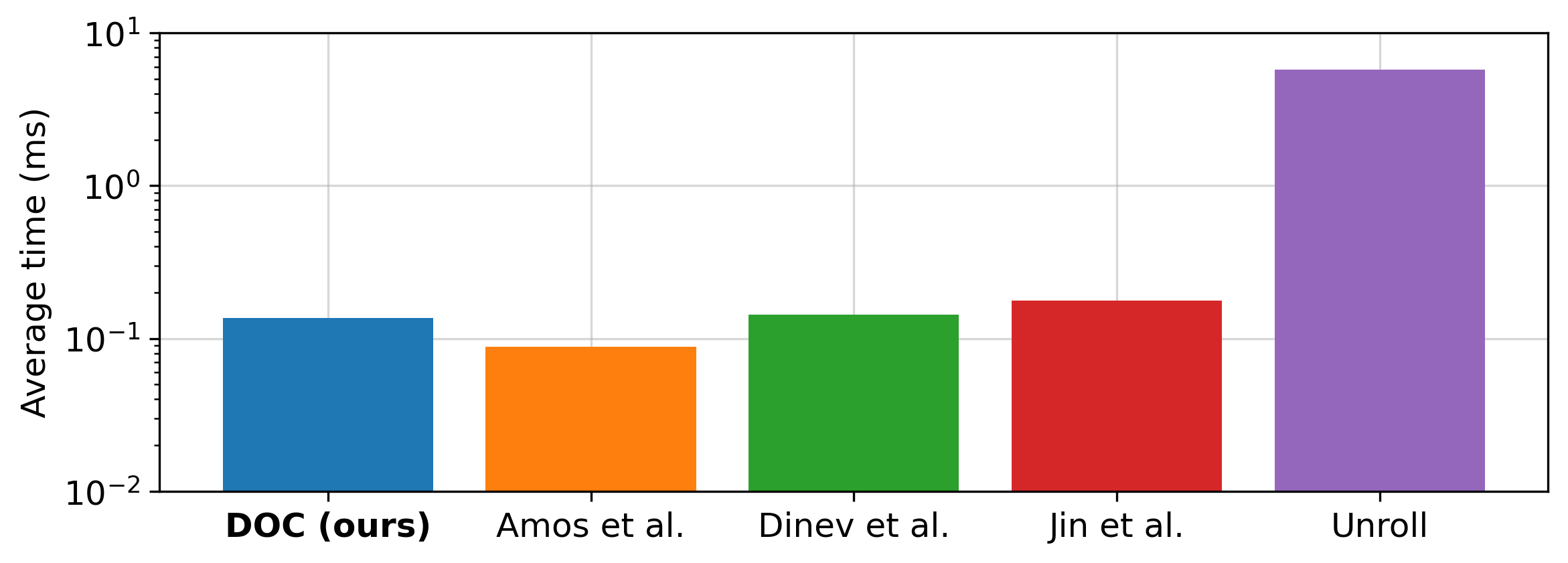}
    \caption{Timing comparison. Note the log scale on the y-axis.}
    \label{fig:dubins_timing}
\end{subfigure}
\begin{subfigure}[b]{0.5\textwidth}
    \centering
    \includegraphics[width=1.0\textwidth]{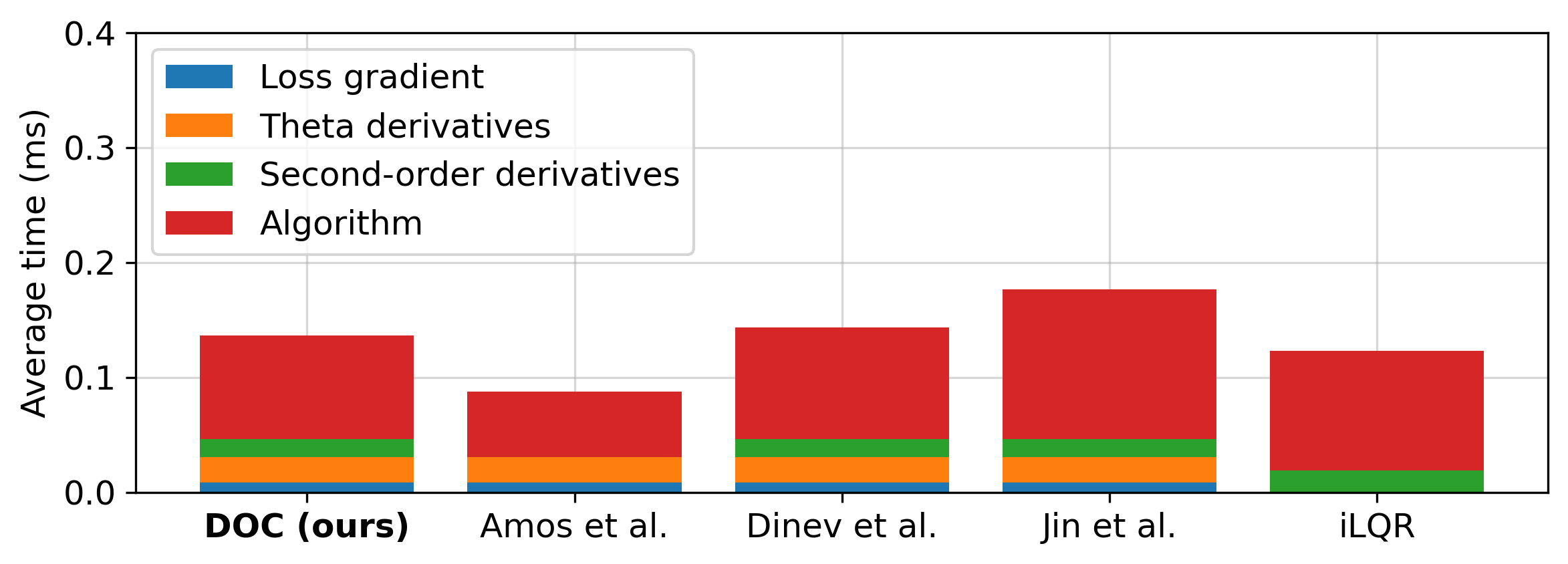}
    \caption{Timing breakdown. iLQR timings are shown for reference.}
    \label{fig:dubins_timing_breakdown}
\end{subfigure}
\caption{Dubins vehicle numerical comparisons.}
\label{fig:dubins_comparisons}
\end{figure}
\vfill
\newpage
\null
\vfill
\begin{figure}[H]
\centering
\begin{subfigure}[b]{0.5\textwidth}
    \centering
    \includegraphics[width=0.8\textwidth]{figures/grad_error_quadrotor.png}
    \caption{Jacobian estimate error (repeated from \cref{fig:grad_error} for reference).}
    \label{fig:quad_error}
\end{subfigure}
\begin{subfigure}[b]{0.5\textwidth}
    \centering
    \includegraphics[width=1.0\textwidth]{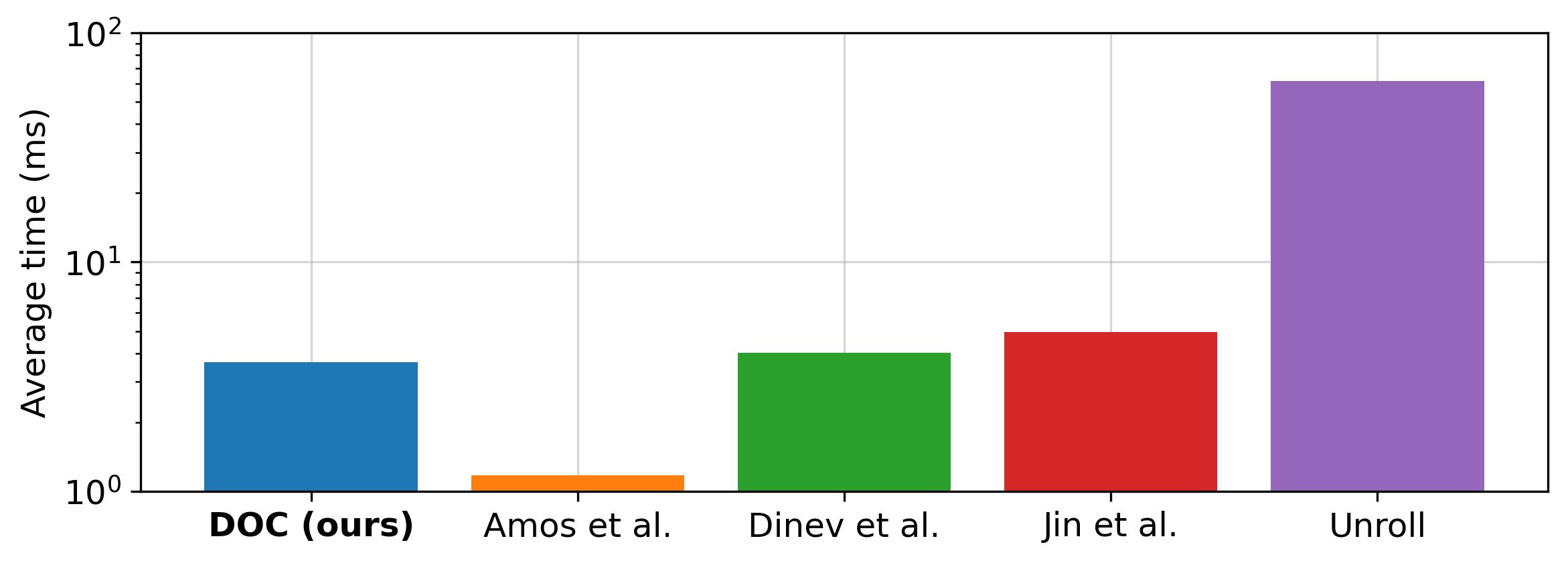}
    \caption{Timing comparison. Note the log scale on the y-axis.}
    \label{fig:quad_timing}
\end{subfigure}
\begin{subfigure}[b]{0.5\textwidth}
    \centering
    \includegraphics[width=1.0\textwidth]{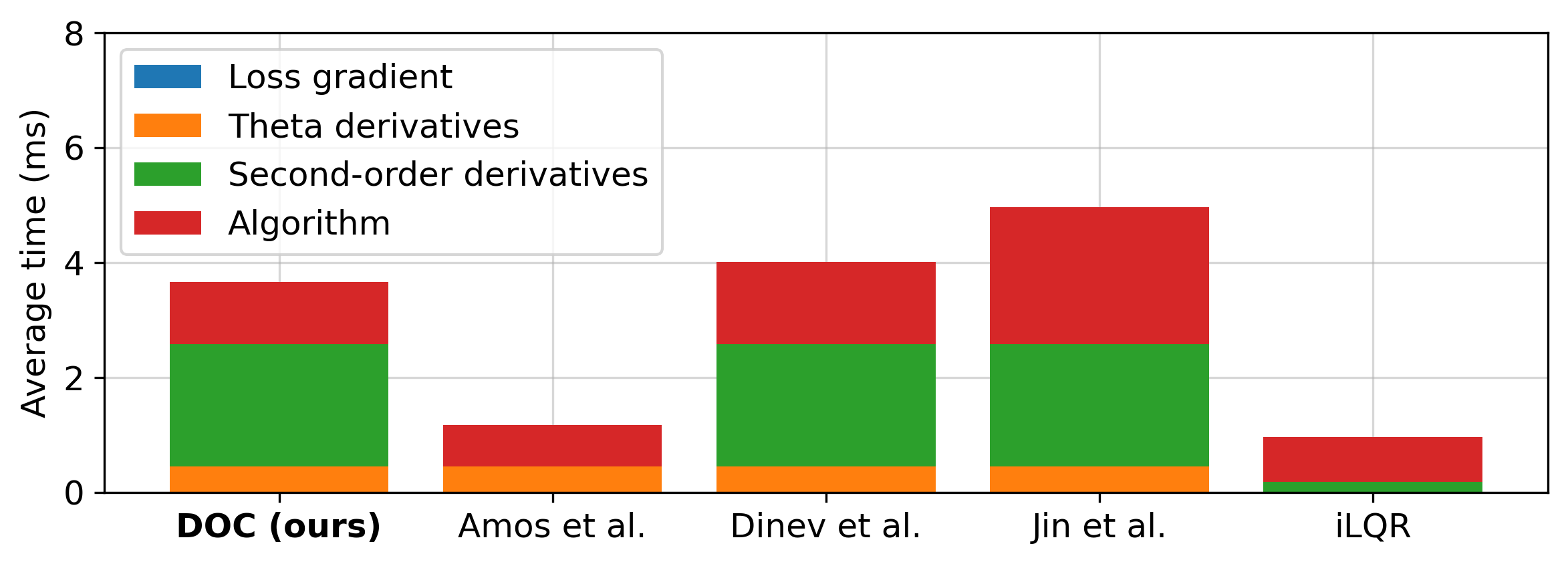}
    \caption{Timing breakdown. iLQR timings are shown for reference.}
    \label{fig:quad_timing_breakdown}
\end{subfigure}
\caption{Quadrotor numerical comparisons.}
\label{fig:quadrotor_comparisons}
\end{figure}
\vfill
\newpage
\vfill
\begin{figure}[H]
\centering
\begin{subfigure}[b]{0.5\textwidth}
    \centering
    \includegraphics[width=0.8\textwidth]{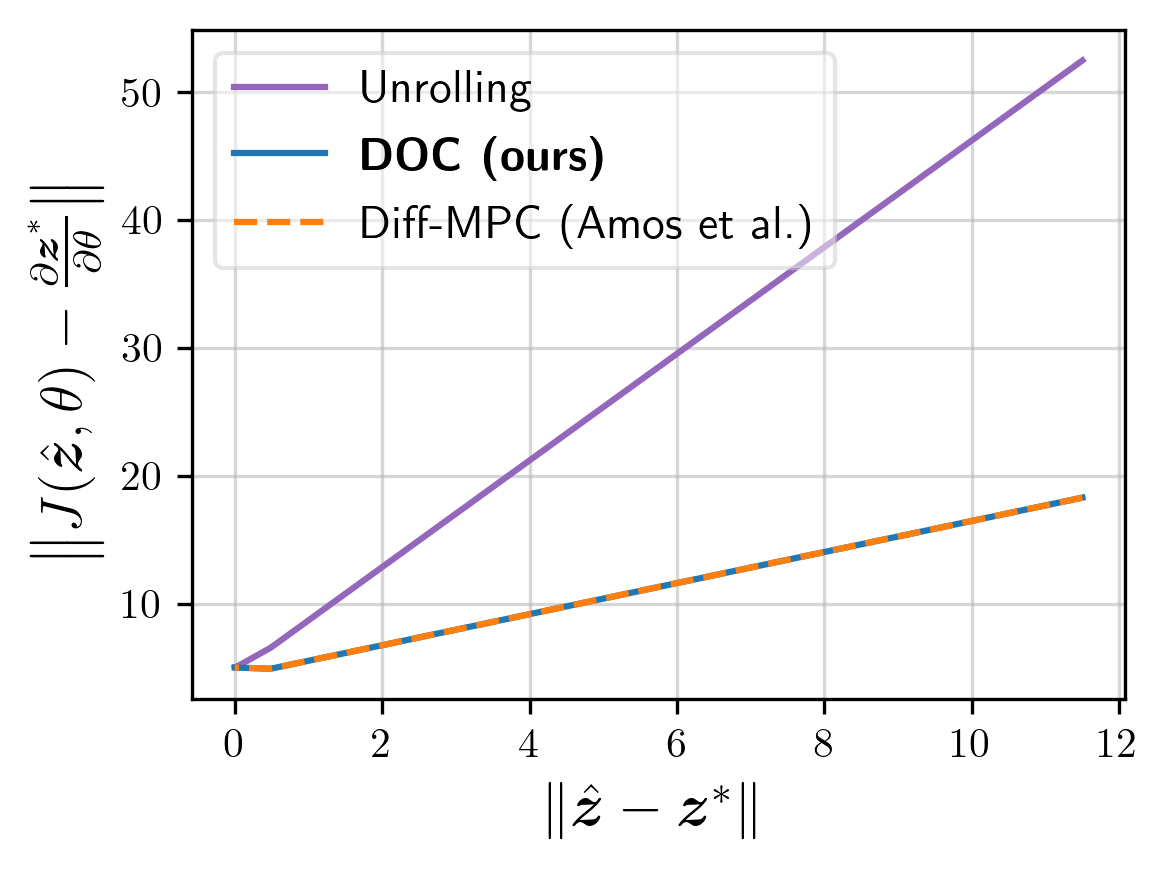}
    \caption{Jacobian estimate error.}
    \label{fig:robot_arm_error}
\end{subfigure}
\begin{subfigure}[b]{0.5\textwidth}
    \centering
    \includegraphics[width=1.0\textwidth]{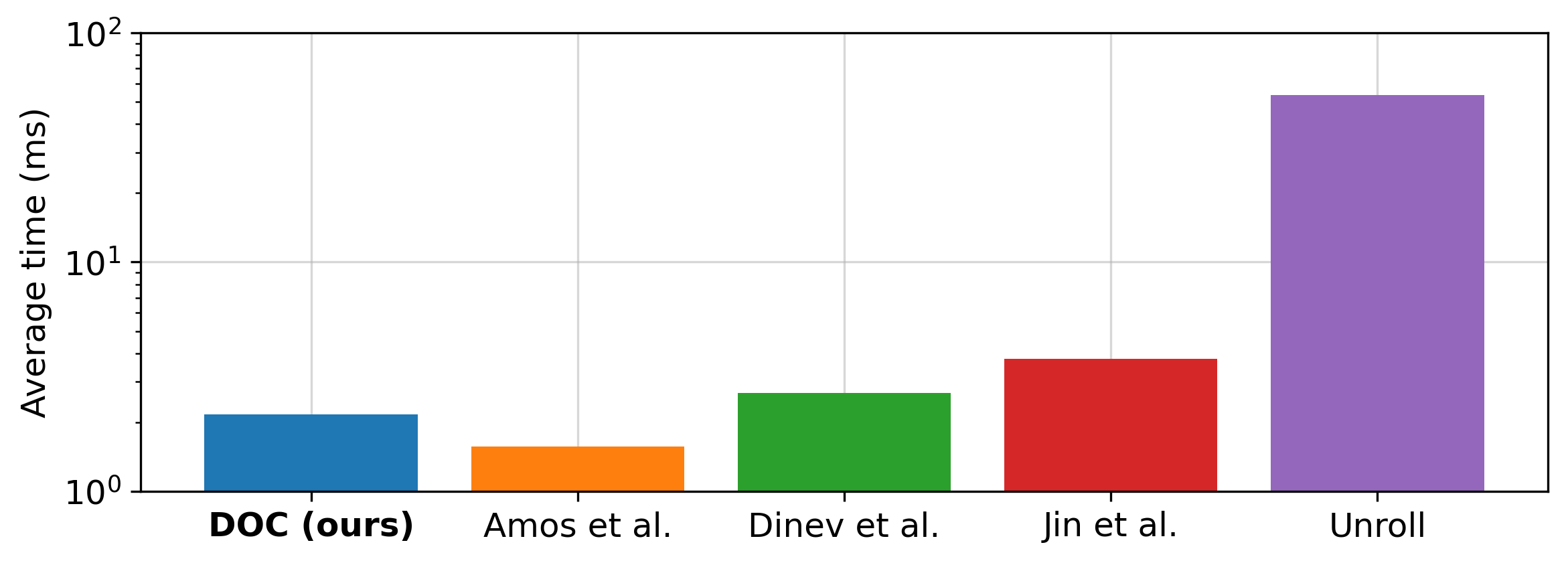}
    \caption{Timing comparison. Note the log scale on the y-axis.}
    \label{fig:robot_arm_timing}
\end{subfigure}
\begin{subfigure}[b]{0.5\textwidth}
    \centering
    \includegraphics[width=1.0\textwidth]{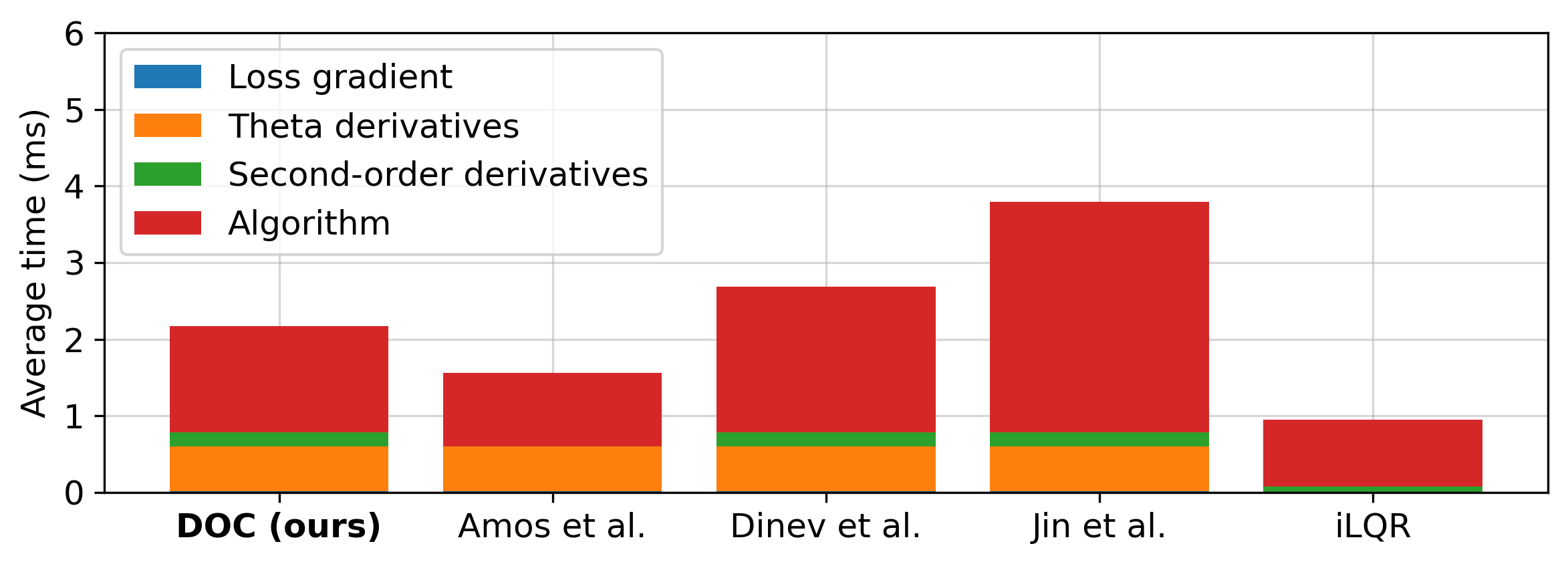}
    \caption{Timing breakdown. iLQR timings are shown for reference.}
    \label{fig:robot_arm_timing_breakdown}
\end{subfigure}
\caption{Robot arm numerical comparisons.}
\label{fig:robot_arm_comparisons}
\end{figure}
\vfill

As Diff-MPC \citep{amos2018differentiable} uses an LQ approximation to the control problem, their algorithm is able to achieve very fast timings. However, this results in inaccurate gradients, which is reflected in the high Jacobian estimate error compared to the proposed DOC algorithm. Note that in \cref{fig:robot_arm_error}, the Diff-MPC and DOC Jacobian errors align --- this is due to modeling the underlying dynamics of the system as a 6-dimensional double integrator system. This is a linear system, and therefore, the second-order dynamics derivatives are zero --- the nonlinearities in this system appear in the cost function design.

We highlight the fact that our algorithm is faster, as well as more memory efficient, compared to the work of \citet{dinev2022differentiable} due to the smart accumulation of the gradient during the forward pass of \cref{alg:doc}. This prevents needing to store intermediate vectors $\delta{x_t}$, etc. The PDP algorithm of~\citet{jin2020pontryagin} produces accurate gradients, but is slower than the other methods due to needing to solve a matrix control system.

Since the MuJoCo dynamics are not compatible with autodiff, comparisons on those systems are omitted. However, we provide a timing comparison between NT-MPC and the proposed DT-MPC on these two systems in \cref{fig:mujoco_timings}. Our proposed method is comparable in runtime with the baseline while vastly improving task performance and safety.
\begin{figure}[h]
\centering
\begin{subfigure}[b]{0.24\textwidth}
    \centering
    \includegraphics[width=0.99\textwidth]{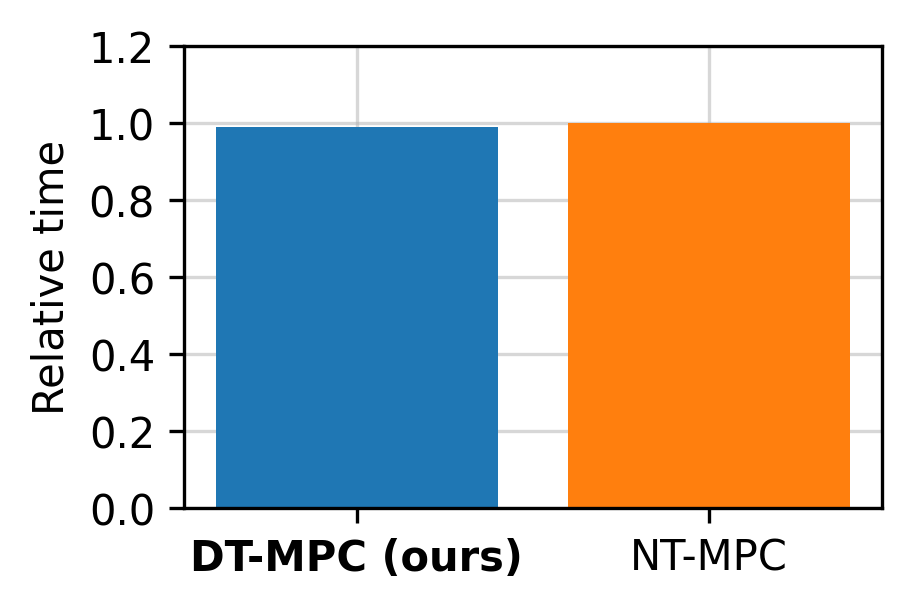}
    \caption{Cheetah system.}
\end{subfigure}\hfill\begin{subfigure}[b]{0.24\textwidth}
    \centering
    \includegraphics[width=0.99\textwidth]{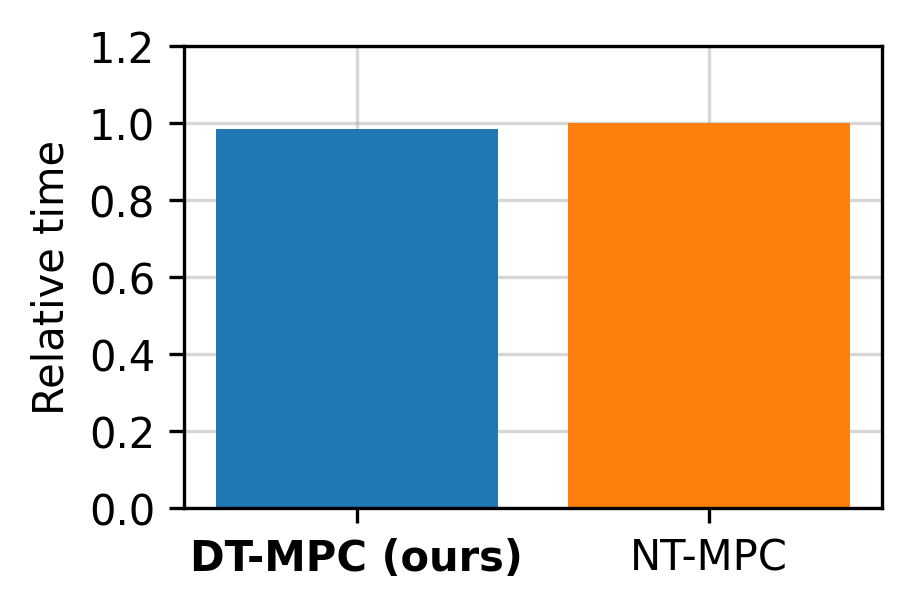}
    \caption{Quadruped system.}
\end{subfigure}
\caption{Timing comparison on MuJoCo systems. Values are normalized with the NT-MPC time corresponding to a value of 1 to show the relative speedup/slowdown of our approach.}
\label{fig:mujoco_timings}
\end{figure}

\section{Further Experimental Detail}
\label{appendix:experiments}

This section provides supplementary detail for the experiments performed in \cref{sec:experiments}. In particular, we define the varying parameterizations of each controller (i.e., what is allowed to adapt during MPC) to highlight the generality of the proposed method.

For both the nominal and ancillary controllers, we use DBaS-DDP~\citep{almubarak2022safety} and the control-limited DDP solver for handling box control limits~\citep{tassa2014control}, dropping the second-order dynamics terms (equivalent to the iLQR approximation). For all experiments, we use the (relaxed) inverse barrier. Additionally, we consider a ``computational budget'' of 10 iterations. This means NT-MPC can optimize both controllers for the full 10 iterations. However, our method optimizes the lower-level controllers for 9 iterations and uses the last to calculate the parameter gradients as long as either controller converged. Convergence in this sense is declared if the cost does not improve between iterations more than \num{1e-3}.


Of particular note is the choice of learning rate in \cref{alg:dt_mpc}. We find that learning rates smaller than $\eta = \num{1e-3}$ tend to suffer in both task completion as well as obstacle avoidance. On the other hand, we have found that a relatively large learning rate of $\eta = \num{1e-2}$ enables fast response to disturbances, especially as the system approaches an obstacle. To improve the efficiency of the method, we adopt a Nesterov momentum scheme~\citep{nesterov2003introductory}, but it should be noted that vanilla gradient descent works well in practice.

For all experiments, the nominal controller has a cost function dependent on the task being solved, while the ancillary controller has the general form $\ell = \norm{x_k - \bar{x}_k}_Q^2 + \norm{u_k - \bar{u}_k}_R^2 + q_b b_k^2$ and $\phi = \norm{x_N - \bar{x}_N}_Q^2 + q_b b_N^2$, with weights generally initialized to all ones.
To avoid invalid control solutions (i.e., for DDP $R$ must be positive definite), we constrain the parameters (when applicable) through projected gradient descent to ensure: $Q_{ii} \geq 0$, $R_{ii} \geq \num{1e-4}$, ${q_b \in [0, 1]}$, $\gamma \in [-1, 1]$, and $\alpha \geq 0$.

\subsection{Dubins Vehicle}
The Dubins vehicle is a nonlinear system with three states (xy position and yaw angle $\theta$) and two controls (linear velocity $v$ and angular velocity $\omega$). The controls are constrained such that $|v|~\leq$~\SI{10}{\meter\per\second} and $|\omega|~\leq~\pi$~\SI{}{\radian\per\second}. The dynamics are discretized with a time step of $\Delta t = 0.01$. The controller horizon is $N = 50$ corresponding to a planning horizon of \SI{0.5}{\second} long. The task is run for $H = 300$ time steps, or until success or failure occurs. Success in \cref{table:percentages} is defined as arriving within a \SI{0.25}{\meter} circle of the target, while failure is defined as colliding with an obstacle.

The nominal cost is fixed and given by $\bar{\ell} = \norm{x_k - x_\text{target}}_{\bar{Q}}^2 + \norm{u_k}_{\bar{R}}^2 + \bar{q}_b b_k^2$ and $\bar{\phi} = \norm{x_N - x_\text{target}}_{\bar{Q}_f}^2 + \bar{q}_b b_N^2$,
where ${x_\text{target} = (10, 10, \pi/4)}$, ${\bar{Q} = \diag(1, 1, 0)}$, ${\bar{R} = \diag(1, 1)}$, ${\bar{Q}_f = \diag(1000, 1000, 1000)}$, $\bar{q}_b = 1$. The ancillary cost is initialized to all ones and allowed to adapt through minimization of \cref{eq:dt_mpc_loss}. For this experiment, we use the inverse barrier (relaxed inverse barrier with $\alpha = 0$) and fix $\gamma = 0$.


\subsection{Quadrotor}

The obstacle field in \cref{fig:quadrotor_comparison} is generated as follows. 30 spherical obstacles are created by sampling a center point from the range $[0, 10]$~\SI{}{\meter} in all three axes and a radius from the range $[0.5, 1.5]$~\SI{}{\meter}. An additional obstacle with radius \SI{1.5}{\meter} is placed at the center point $(5, 5, 5)$~\SI{}{\meter} between the origin and the target to avoid trivial solutions.

Success for this task is defined as arriving within a \SI{0.5}{\meter} sphere of the target, while failure is defined as colliding with an obstacle or leaving the bounds $[-2, 12]$~\SI{}{\meter} in any of the three axes.

The dynamics model is adapted from \citet{sabatino2015quadrotor} and contains 12 states --- xyz position, orientation expressed as Euler angles, and linear and angular velocities --- and four controls --- thrust magnitude and pitching moments. For simplicity, unity parameters (\SI{1}{\kilogram} mass, identity inertia matrix, etc.) are adopted. The thrust is limited to the range $[0, 50]$ \SI{}{\newton} and the pitching moments to $[-10, 10]$ \SI{}{\newton\meter}. The dynamics are discretized with $\Delta t = 0.02$ and the MPC horizon is $N = 50$ for both controllers corresponding to a \SI{1}{\second} planning horizon. The task is run for $H = 300$ time steps, or until success or failure occurs.

Similar to the Dubins vehicle experiment, the nominal cost is fixed and given by $\bar{\ell} = \norm{x_k - x_\text{target}}_{\bar{Q}}^2 + \norm{u_k}_{\bar{R}}^2 + \bar{q}_b b_k^2$ and $\bar{\phi} = \norm{x_N - x_\text{target}}_{\bar{Q}_f}^2 + \bar{q}_b b_N^2$,
where $\bar{Q} = \diag(\ones(12))$, $\bar{R} = \diag(\ones(4))$, $\bar{Q}_f = 1000 * \diag(\ones(12))$, and $\bar{q}_b = 1$. The target state is $x_\text{target} = (10, 10, 10, 0, 0, 0, 0, 0, 0, 0, 0, 0)$. The ancillary cost is initialized to all ones and adapted online through minimization of \cref{eq:dt_mpc_loss}. The inverse barrier (relaxed inverse barrier with $\alpha = 0$) is used and $\gamma = 0$ is fixed.

\subsection{Robot Arm}

This system has 12 states corresponding to the orientation $\theta$ (pitch and yaw angles) and angular velocities $\dot{\theta}$ of each of the sections of the arm, and six controls corresponding to the torques $\ddot{\theta}$ applied at each joint. The torques are limited to $|\ddot{\theta}| \leq 10$ \SI{}{\newton\meter}. The three links of the arm have lengths \SI{1}{\meter}, \SI{1.5}{\meter}, and \SI{1}{\meter}, respectively. The obstacles visualized in \cref{fig:robot_arm_env} are fixed with centers $(1, 0)$, $(1, 1.5)$, $(1, -1.5)$, $(2, -2)$, and $(2, 2)$~\SI{}{\meter}, each with radius \SI{0.5}{\meter} --- a \SI{0.5}{\meter} gap exists on either side of the central obstacle that the arm can safely pass through. At the start of each trial, the arm begins in a random feasible starting orientation which is generated by sampling angles from the range $[-\pi, \pi]$~\SI{}{\radian} and rejecting configurations with the arm below the xy-plane or conflicting with obstacles.

The dynamics are discretized with a time step of $\Delta t = 0.02$, and the MPC time horizon for both controllers is 50 time steps, corresponding to a \SI{1}{\second} planning horizon. The task is run for at most $H = 400$ time steps, or until success (the end effector enters a \SI{0.25}{\meter} sphere of the target location) or failure (any link of the arm collides with an obstacle \emph{or} the arm goes below the xy-plane) occurs.

The nominal cost function here is reparameterized in terms of the end effector position $e_k$. The running cost is given as $\bar{\ell} = \norm{e_k - e_\text{target}}_{\bar{Q}}^2 + \norm{u_k}_{\bar{R}}^2 + \bar{q}_b b_k^2$ and the terminal cost $\bar{\phi} = \norm{e_N - e_\text{target}}_{\bar{Q}_f}^2 + \bar{q}_b b_N^2$, with $\bar{Q} = 100 * \diag(\ones(3))$, $\bar{R} = 100 * \diag(\ones(6))$, $\bar{Q}_f = 10000 * \diag(\ones(3))$, $\bar{q}_b = \num{1e-3}$, and $e_\text{target} = (2, 0, 1)$. The ancillary cost is the same as described previously and initialized to all ones, except for the barrier weight which is initialized as $q_b = \num{1e-3}$.

\subsection{Cheetah}
The cheetah system has 18 states and 6 controls with $\Delta t = 0.01$ and control limits $u_i \in [-1, 1]$. The controller plans for $N = 50$ time steps with an allotted task time of $H = 300$. The nominal cost is given as $\bar{\ell} = \bar{q} (p_x^{(k)} - 5)^2 + \norm{u_k}_{\bar{R}}^2 + \bar{q}_b b_k^2$ and $\bar{\phi} = \bar{q}_f (p_x^{(N)} - 5)^2 + \bar{q}_b b_N^2$, where $p_x^{(k)}$ is the x-position at time $k$, and $\bar{q} = 1$, $\bar{R} = 0.01 * \diag(\ones(6))$, $\bar{q}_f = 100$, and $q_b = 1$. Notably, we fix the terminal cost but allow the remaining parameters to adapt online during MPC --- this allows the DT-MPC to improve the task completion percentage without sacrificing safety. The ancillary controller is initialized to all ones. Additionally, we use the relaxed barrier with initialization $\gamma=0, \alpha = 0.1$ and allow the parameters to adapt along with the cost function weights. The use of the relaxed barrier here is beneficial as small violations of the pitch angle constraint are acceptable as long as the cheetah does not flip over. As stated in the main text, the loss for both controllers is $L = \norm{\boldsymbol{p}_x^* - \bar{\boldsymbol{p}}_x}_2^2 + \norm{\boldsymbol{b}^*}_2^2$, motivating safe task completion.

\subsection{Quadruped}
The quadruped system is a complex robotic system with 56 states and 12 controls. Each leg is a simplified biological model with an extending/contracting tendon. The controls for each leg consist of the yaw angle, lift, and extension of the leg. The dynamics are discretized with $\Delta t = 0.005$ and the MPC horizon is $N = 50$. The task is run for $H = 400$ time steps. 

The nominal cost is similar to the cheetah experiment, with $\bar{\ell} = \bar{q} (p_x^{(k)} - 2.5)^2 + \norm{u_k}_{\bar{R}}^2 + \bar{q}_b b_k^2$ and $\bar{\phi} = \bar{q}_f (p_x^{(N)} - 2.5)^2 + \bar{q}_b b_N^2$, and the weights given as $\bar{q} = 1$, $\bar{R} = 0.01 * \diag(\ones(6))$, $\bar{q}_f = 100$, and $q_b = 1$. The ancillary cost is initialized to all ones, except the barrier weight which is initialized to $q_b = 0.1$. The relaxed inverse barrier is adopted here with initial parameters $\gamma=0, \alpha = 1.0$, and both controllers minimize the loss $L = \norm{\boldsymbol{p}_x^* - \bar{\boldsymbol{p}}_x}_2^2 + \norm{\boldsymbol{b}^*}_2^2$ for adaptation.

\subsection{Hardware Experiment - Robotarium}
For the Robotarium experiment, we adopt the same Dubins vehicle model as used previously for planning and initialize five antagonistic agents who slowly move forwards with random velocity. The dynamics is discretized with $\Delta t = 0.033$ and the MPC horizon is $N = 50$ for an overall planning time of \SI{1.65}{\second} into the future. State constraints are added to ensure the robot stays within the desired operating region (dashed lines of \cref{fig:robotarium}) and avoids the other agents. The controller must run under \SI{33}{\milli\second} or at a rate of about \SI{30}{\hertz} in order for the control to be fast enough for real-time. To ensure this, we take advantage of the compilability of JAX~\citep{jax2018github}. 

For DT-MPC, the nominal parameters are fixed while the ancillary parameters are allowed to adapt. The loss function is chosen as $L = \norm{\boldsymbol{p}_x^* - \bar{\boldsymbol{p}}_x}_2^2 + \norm{\boldsymbol{p}_y^* - \bar{\boldsymbol{p}}_y}_2^2 + \norm{\boldsymbol{b}^*}_2^2$. This choice is due to the fact that the orientation of the robot relative to the nominal trajectory is not very informative for reaching the target. Similar to the cheetah and quadruped examples, we have observed that when the orientation is included in the loss, the DT-MPC controller defaults to a safe control but does not reach the target.

\end{appendices}

\end{document}